\documentclass[11pt]{article}

\usepackage[T1]{fontenc}
\usepackage{amsmath,amssymb,amsfonts,amsthm,mathtools}
\usepackage{mathrsfs}
\usepackage{enumitem}
\usepackage{hyperref}
\usepackage{geometry}

\hypersetup{hidelinks}

\numberwithin{equation}{section}
\hypersetup{
  colorlinks=true,
  linkcolor=blue,
  citecolor=blue,
  urlcolor=blue
}
\theoremstyle{plain}
\newtheorem{theorem}{Theorem}[section]
\newtheorem{proposition}[theorem]{Proposition}
\newtheorem{lemma}[theorem]{Lemma}

\newtheorem{assumption}[theorem]{Assumption}
\newtheorem{definition}[theorem]{Definition}

\newtheorem{remark}{Remark}[section]

\newcommand{\Hh}{\mathsf H}
\newcommand{\Uh}{\mathsf U}
\newcommand{\Q}{\mathsf Q}
\newcommand{\E}{\mathbb E}
\newcommand{\Pp}{\mathbb P}
\newcommand{\R}{\mathbb R}
\newcommand{\N}{\mathcal N}
\newcommand{\Dom}{\mathcal D}
\newcommand{\Ran}{\operatorname{Ran}}
\newcommand{\Tr}{\operatorname{Tr}}
\newcommand{\divnu}{\operatorname{div}_{\nu}}
\newcommand{\EP}{\operatorname{EP}}
\newcommand{\ep}{\operatorname{ep}}
\newcommand{\dd}{\,\mathrm d}
\newcommand{\inner}[2]{\left\langle #1,#2\right\rangle_{\Hh}}

\newcommand{\Lzero}{\mathcal L_0}

\title{Entropy Production and Reversibility Criteria for Stochastic Evolution Equations}

\author{
Jinqiao Duan,\textsuperscript{1} \quad
Ao Zhang\textsuperscript{2}  \quad and \quad
Johannes Zimmer\textsuperscript{3}\\[0.7em]
\small\textsuperscript{1}Department of Mathematics and Department of Physics, Great Bay University,\\[-0.05em]
\small Dongguan, Guangdong 523000, China\\[0.2em]
\small \textsuperscript{2}School of Mathematics and Statistics, HNP-LAMA,
Central South University, \\[-0.05em]
\small Changsha 410083, China\\[0.2em]
\small \textsuperscript{3}School of Computation, Information and Technology,
Technische Universit\"at M\"unchen,\\[-0.05em]
\small Boltzmannstr. 3, 85748 Garching, Germany\\[0.45em]
\small
\href{mailto:duan@gbu.edu.cn}{\texttt{duan@gbu.edu.cn}}\qquad
\href{mailto:aozhang1993@csu.edu.cn}{\texttt{aozhang1993@csu.edu.cn}}
\qquad
\href{mailto:jz@tum.de}{\texttt{jz@tum.de}}
}

\date{}

\begin{document}

\maketitle

\begin{abstract}

This paper develops a path-space theory of entropy production for a class of
stochastic evolution equations on infinite-dimensional Hilbert spaces. Since
such spaces have no canonical Lebesgue reference measure, the usual
finite-dimensional density formulas do not extend directly. We instead work
relative to the invariant Gaussian measure of a reversible
Ornstein--Uhlenbeck reference process. Combining an infinite-dimensional
Girsanov transform, time reversal of the reference process, and the stationary
Fokker--Planck equation relative to the Gaussian measure, we derive an
explicit entropy-production formula in terms of an irreversibility field. On
the natural test class, this field represents the difference between the
forward and reversed nonlinear drifts. Under the standing assumptions,
vanishing entropy production is equivalent to vanishing stationary
probability current, self-adjointness of the generator in the invariant Hilbert space, detailed balance, and invariance of the stationary path law
under time reversal. The reversible case is therefore characterized by a
Gaussian-reference gradient structure for the nonlinear drift.

\end{abstract}

\section{Introduction}
Entropy production quantifies the breaking of time-reversal symmetry in a
stationary stochastic system. If \(\Pp^+_{[0,T]}\) denotes the forward path
law and \(\Pp^-_{[0,T]}\) its time reversal, the relative entropy
\(\mathcal H(\Pp^+_{[0,T]}\mid\Pp^-_{[0,T]})\) measures the statistical
distinguishability of the two directions of time. Its rate therefore links
thermodynamic dissipation to probabilistic irreversibility, detailed balance,
and stationary probability currents. This path-space viewpoint underlies
both diffusion-based entropy-production formulas and fluctuation relations
for stochastic dynamics~\cite{LebowitzSpohn1999, QW99}.

The purpose of this paper is to establish an entropy-production formula for
a class of stochastic evolution equations
\begin{equation}
\label{eq:intro-spde}
    \dd X_t=(AX_t+F(X_t))\,\dd t+B\,\dd W_t,
    \qquad X_t\in\Hh.
\end{equation}
Here \(\Hh\) and \(\Uh\) are real separable Hilbert spaces, \(W\) is a
cylindrical Wiener process on \(\Uh\), \(A\) generates an exponentially
stable semigroup, \(B\) is a Hilbert--Schmidt operator from \(\Uh\) to \(\Hh\) and \(F\) is a nonlinear perturbation. The central idea is to
use the invariant Gaussian measure \(\nu\) of the Ornstein--Uhlenbeck equation
\begin{equation}\label{eq:ou-reference}
    \dd Z_t = AZ_t \, \dd t + B \, \dd W_t.
\end{equation}
We assume that this Ornstein--Uhlenbeck process is reversible with respect to
\(\nu\) and consider a stationary solution of~\eqref{eq:intro-spde} whose
invariant measure \(\mu\) is absolutely continuous with respect to \(\nu\):
\[
    \mu(\dd x)=\rho(x)\nu(\dd x),
\]
with a strictly positive density \(\rho\). The precise hypotheses in
Section~\ref{Sec2} include Gaussian Sobolev regularity of \(\rho\), a
noise-range condition on the drift, and sufficient strong-solution
regularity for the pathwise time-reversal calculation.

Let
\[
    V:=2F-\Q \nabla\log\rho, \qquad \text{with}~\Q:=BB^*.
\]
We call \(V\) the irreversibility field. On the admissible test class used in
the formal adjoint calculation, it represents the difference between
the forward and reversed nonlinear drifts; the conventional irreversible
component is \(V/2\). The regularity assumptions ensure that \(V\) belongs to
the Cameron--Martin noise space \(\Hh_{\Q}=\Ran(\Q^{1/2})\), equipped with the
minimum-control norm \(\|\cdot\|_{\Q^{-1}}\). Under the standing assumptions,
the forward and reversed path laws are equivalent on every finite time
interval, and the local and total entropy production in Theorem~\ref{thm:main}  are
\[
    \ep(x)=\frac12\|V(x)\|_{\Q^{-1}}^2,
    \qquad
    \EP
    =
    \frac12
    \int_{\Hh}
    \|V(x)\|_{\Q^{-1}}^2\,\mu(\dd x).
\]
The first identity holds for \(\mu\)-almost every \(x\).
The corresponding Gaussian-reference probability current is
\[
    J_\mu=\rho F-\frac12\Q \nabla\rho.
\]
It satisfies \(J_\mu=\rho V/2\), while the stationary
Fokker--Planck identity \eqref{eq:FP} gives \(J_\mu\in\Dom(\divnu)\) and
\(\divnu J_\mu=0\) in \(L^2(\nu)\).

Within this framework, vanishing entropy production is equivalent to
\(V=0\), to \(J_\mu=0\), and to the Gaussian-reference gradient condition
\[
    F=\frac12\Q \nabla\log\rho.
\]
These conditions are further equivalent to self-adjointness of the closed
generator in \(L^2(\mu)\), detailed balance of the Markov semigroup, and
invariance of the stationary path law under time reversal in Theorem \ref{thm:reversibility-criteria}. The equivalence
between generator symmetry, detailed balance, and pathwise reversibility is
standard for symmetric Markov processes~\cite[Theorem 4.3.3]{jiang2004}; our
contribution is to identify these properties with the explicit Gaussian
current and Cameron--Martin energy above.

The proof adapts the finite-dimensional path-reversal strategy
of~\cite{QW99} to the Gaussian reference setting. An infinite-dimensional
Girsanov transform first removes the nonlinear drift. Detailed balance for
the Ornstein--Uhlenbeck semigroup then yields a weighted endpoint identity
under time reversal. The remaining correction is computed using the closed
Gaussian divergence. Because an \(\Hh\)-valued strong solution is assumed, the
reversal of the stochastic line integral can be obtained directly from
left- and right-point Riemann sums and their trace covariation; no
Galerkin-level Stratonovich integral is introduced. These ingredients give
the finite-time Radon--Nikodym derivative of the forward law with respect to
the reversed law, from which the local and total entropy-production formulas
follow.

The framework is deliberately more specialized than the general
infinite-dimensional time-reversal theories
in~\cite{FollmerWakolbinger1986,MilletNualartSanz1989}: it requires a reversible
Gaussian Ornstein--Uhlenbeck reference process, a nonlinear perturbation in
the noise range, and an \(\Hh\)-valued strong solution with sufficient
differentiability. In return, it provides an explicit Gaussian-reference
current and a Cameron--Martin energy identity while allowing \(A\) to be
unbounded and the
noise covariance to have no bounded inverse. The results do not claim to
construct a strong reversed SPDE in full generality, nor do they treat the
singular forward--backward regime outside the noise range.

For finite-dimensional diffusions, Qian and Wang related entropy production
to circulation and to the squared norm of the irreversible drift on a
Riemannian manifold~\cite{QW99}. Lebowitz and Spohn identified the
corresponding forward--backward path action as the quantity entering a
Gallavotti--Cohen-type fluctuation symmetry~\cite{LebowitzSpohn1999}.
More recently, Da Costa and Pavliotis treated stationary diffusions with
possibly degenerate noise. Their results show that equivalence of the
forward and reversed path laws is governed by a noise-range condition; in
the regular branch the entropy production is expressed through a
Moore--Penrose inverse, whereas failure of the range condition can produce
mutually singular path laws and infinite entropy production~\cite{DaCostaPavliotis2023}.

A complementary line of research comes from stochastic thermodynamics.
Seifert formulated total entropy production along individual Langevin and
Markov-jump trajectories and derived an integral fluctuation
theorem~\cite{Seifert2005}. Chetrite and Gaw\k{e}dzki subsequently developed a
unified treatment of fluctuation relations for finite-dimensional diffusion
processes by comparing the original and time-reversed
dynamics~\cite{ChetriteGawedzki2008}. The same path-space asymmetry also has an
inferential interpretation: Rold\'an and Parrondo showed that the relative
entropy rate between a stationary time series and its time reversal can be
used to estimate dissipation, whereas coarse-grained or partial observations
generally yield only a lower bound~\cite{RoldanParrondo2010}.

For continuous stochastic dynamics in full phase space, physical time
reversal must also account for the parity of each coordinate: positions are
even, whereas momenta are odd. Building on the excess--housekeeping and
nonadiabatic--adiabatic decompositions
in~\cite{EspositoVanDenBroeck2010FT,  HatanoSasa2001, JackKaiserZimmer2017,  SpeckSeifert2005,
VanDenBroeckEsposito2010II}, Spinney and Ford decomposed the drift and
probability current into reversible and irreversible parts and expressed
the mean total entropy-production rate as a nonnegative,
diffusion-weighted quadratic functional of the irreversible
current~\cite{SpinneyFordOddEven2012,SpinneyFord2012}. They further wrote the
trajectory entropy production as
\(\Delta S_{\mathrm{tot}}=\Delta S_1+\Delta S_2+\Delta S_3\): the first
two contributions obey integral fluctuation theorems and have nonnegative
means, whereas the third can arise in the presence of odd variables and a
parity-asymmetric stationary density, need not obey an integral fluctuation
theorem, and may have either sign in the mean~\cite{SpinneyFord2012}. The
present paper uses the
ordinary identity-parity path reversal, so these full-phase-space results
serve as motivation rather than as a direct finite-dimensional case of our
theory.

Entropy production has also been formulated directly for stochastic density
fields of interacting particles. Starting from Dean's exact density
equation~\cite{Dean1996}, Brossollet and Biroli used Onsager--Machlup path
weights to derive trajectory-wise and steady-state entropy production for
Dean-type and coarse-grained density theories~\cite{BrossolletBiroli2026}.
For pair interactions, the stationary rate reduces to a formula involving stationary
one- and two-point density correlations and agrees with particle-level and
Doi--Peliti calculations; the latter provide a complementary microscopic
field representation~\cite{PruessnerGarciaMillan2025}. Related active-field
work resolves irreversibility in real and Fourier space through
forward--backward field-path probabilities~\cite{NardiniEtAl2017}. These
studies also emphasize that multiplicative noise, discretization, and
spurious-drift terms must be treated consistently~\cite{CatesEtAl2022}.
Since Dean/DDFT dynamics has conserved, state-dependent noise, it is
structurally related to, but not covered by, the additive Gaussian-reference
framework studied here.

A related development concerns universal restrictions imposed by entropy
production on current fluctuations. Barato and Seifert proposed the
thermodynamic uncertainty relation for steady-state Markov networks, which
expresses a trade-off between current precision and thermodynamic
cost~\cite{BaratoSeifert2015}. Gingrich, Horowitz, Perunov, and England then
proved a quadratic dissipation bound for the current large-deviation
function of finite-state continuous-time Markov jump processes, from which
the long-time uncertainty relation follows~\cite{GingrichEtAl2016}. These
results provide important motivation for the present path-law formulation
and current-based interpretation, but they do not by themselves yield the
infinite-dimensional Gaussian-reference identity established here.

Time reversal itself has a broader literature. The reverse drift of a
finite-dimensional diffusion was characterized under regularity assumptions
by Haussmann and Pardoux~\cite{HaussmannPardoux1986}. Infinite-dimensional
diffusion reversal was subsequently studied under local finite-entropy
conditions by F\"ollmer and Wakolbinger~\cite{FollmerWakolbinger1986},
while Millet, Nualart, and Sanz used techniques from the stochastic calculus
of variations~\cite{MilletNualartSanz1989}. These
works provide essential reverse-process theory, but their objectives and
hypotheses differ from the explicit Gaussian-reference entropy-production
identity developed below.

The analytic theory needed in infinite dimensions is well developed in
several complementary directions. Stochastic evolution equations,
Hilbert-space Girsanov transformations, and invariant measures are treated
systematically in~\cite{DZ92,DZ96}. Symmetry and generators of
Hilbert-space Ornstein--Uhlenbeck semigroups were characterized by
Chojnowska-Michalik and Goldys~\cite{MG02}; Gaussian Sobolev spaces,
integration by parts, and closed divergence operators are developed
in~\cite{Bogachev1998,Bogachev2018}. The general relation between symmetric
Markov semigroups, detailed balance, and reversible stationary processes is
part of Dirichlet-form theory~\cite{FukushimaOshimaTakeda}. Girsanov
densities and trace corrections also occur in the Onsager--Machlup analysis
of stochastic evolution equations~\cite{BardinaRoviraTindel2003}, but that
small-tube problem is different from the forward--backward path-law entropy
considered here.

The paper is organized as follows. Section~\ref{Sec2} introduces the Gaussian
reference framework, states the entropy-production and reversibility
theorems, and gives a diagonal strong-solution example. Section~\ref{Sec3} proves the
Girsanov, Ornstein--Uhlenbeck time-reversal, Gaussian-divergence, path-law,
entropy-production, and reversibility results. The final section summarizes
the conclusions.

\section{State space, assumptions, and main results}\label{Sec2}

Let \(\Hh\) and \(\Uh\) be real separable Hilbert spaces. The state-space
inner product and norm in \(\Hh\) are denoted by
\(\inner{\cdot}{\cdot}\) and \(\|\cdot\|_{\Hh}\), respectively. We write
\(\mathcal L(\Hh)\) for the bounded operators on \(\Hh\),
\(\mathcal L_1(\Hh)\) for the trace-class operators on \(\Hh\), and
\(\mathcal L_2(\Uh,\Hh)\) for the Hilbert--Schmidt operators from \(\Uh\)
to \(\Hh\).

On a filtered probability space \((\Omega,\mathcal F,(\mathcal F_t)_{t\geq0},\Pp)\), \(W\) is a cylindrical Wiener process on
\(\Uh\): for \(u,v\in\Uh\),
\[
    \E[W_t(u)W_s(v)]
    =(t\wedge s)\langle u,v\rangle_{\Uh}.
\]
Equivalently, \(W_t=\sum_k\beta_k(t)f_k\) formally, where
\(\{f_k\}\) is an orthonormal basis of \(\Uh\) and the \(\beta_k\) are
independent standard Brownian motions. If
\(B\in\mathcal L_2(\Uh,\Hh)\), then \(BW\) is an \(\Hh\)-valued Wiener
process with covariance \(\Q:=BB^*\in\mathcal L_1(\Hh)\).
We use the standard Hilbert-space stochastic-integration conventions for
cylindrical Wiener noise; see~\cite{DZ92}.

For a probability measure \(\lambda\) on \(\Hh\) and \(1\leq p<\infty\),
\(L^p(\lambda)\) denotes the usual space of measurable functions satisfying
\[
    \|f\|_{L^p(\lambda)}^p
    :=\int_{\Hh}|f(x)|^p\,\lambda(\dd x)<\infty.
\]
In particular,
\[
    \langle f,g\rangle_{L^2(\lambda)}
    :=\int_{\Hh}f(x)g(x)\,\lambda(\dd x).
\]
We use \(\lambda=\nu\) and \(\lambda=\mu=\rho\nu\) below; the notation
\(L^p(\lambda;E)\) has its standard Bochner-space meaning.

\subsection{Analytic assumptions}

We begin by fixing the linear Gaussian reference dynamics and its covariance.

\begin{assumption}[Reference dynamics]\label{ass:AFQ}
The following conditions hold.
\begin{enumerate}[label=\textup{(\roman*)}]
\item The operator \(A:\Dom(A)\subset\Hh\to\Hh\) generates a strongly continuous
semigroup \(S(t)\) of negative type: there are \(M\geq1\) and \(\omega>0\)
such that
\[
    \|S(t)\|_{\mathcal L(\Hh)}\leq Me^{-\omega t},\quad t\geq 0.
\]
\item The cylindrical Wiener process \(W\) is on \(\Uh\),
the operator \(B\in\mathcal L_2(\Uh,\Hh)\), and
\[
    \int_0^\infty \Tr\left[S(r)\Q S^*(r)\right]\,\dd r<\infty,
\]
where \(\Q:=BB^*\).
\end{enumerate}
\end{assumption}

Under Assumption~\ref{ass:AFQ}, the Ornstein--Uhlenbeck process has the centered Gaussian
invariant measure
\[
    \nu=\mathcal N(0,Q_\infty), \quad \text{with}~Q_\infty:=\int_0^\infty S(r)\Q S^*(r)\,\dd r,
\]
characterized by
\[
    \int_{\Hh}e^{i\inner{x}{h}}\,\nu(\dd x)
    =\exp\left(-\frac12\inner{Q_\infty h}{h}\right),
    \qquad h\in\Hh;
\]
see~\cite[Theorem~6.2.1]{DZ96}. Here \(\Q\) and \(Q_\infty\) are symmetric, nonnegative, and trace class.

More generally, a standard sufficient condition for existence of an
invariant Gaussian measure for the stable linear equation \eqref{eq:ou-reference} is
\[
    \int_0^\infty
    \|S(r)B\|_{\mathcal L_2(\Uh,\Hh)}^2\,\dd r<\infty.
\]
The exponential stability and Hilbert--Schmidt hypotheses above are
convenient sufficient conditions for this
requirement~\cite{Bogachev2018,DZ92}. 

Entropy production is measured only along directions accessible to the
noise; the following space and norm encode this restriction.

\begin{definition}[Cameron--Martin noise space and energy norm]
\label{def:CM-space}
Let \(B^\dagger\) denote the minimal-norm inverse of \(B\) on \(\Ran(B)\),
that is, the restriction of the Moore--Penrose inverse to \(\Ran(B)\) . The Cameron--Martin noise space is
\[
    \Hh_{\Q}
    :=
    \Ran(\Q^{1/2})
    =
    \Ran(B).
\]
For \(v\in\Hh_{\Q}\), define
\begin{equation}\label{eq:noise-energy-norm}
    \|v\|_{\Q^{-1}}
    :=
    \|B^\dagger v\|_{\Uh}
    =
    \inf\bigl\{\|u\|_{\Uh}:Bu=v\bigr\}.
\end{equation}
Equivalently, this is
\(\|\Q^{\dagger/2}v\|_{\Hh}\), where
\(\Q^{\dagger/2}\) is the generalized inverse of \(\Q^{1/2}\).
\end{definition}
Thus the notation \(\|\cdot\|_{\Q^{-1}}\) denotes the noise energy norm.
Equipped with this norm, \(\Hh_{\Q}\) is a Hilbert space isometric to
\((\ker B)^\perp\). The pseudoinverse norm identity follows from the polar
decomposition of \(B\), the operator-range identity follows from Douglas'
range theorem~\cite{Douglas1966}, and the minimum-control property follows by
orthogonal projection onto \((\ker B)^\perp\).

The qualifier ``noise'' is important: \(\Hh_{\Q}\) is the Cameron--Martin
space of a noise increment with covariance \(\Q\), whereas the
Cameron--Martin space of the invariant measure \(\nu\) is
\(\Ran(Q_\infty^{1/2})\); the two spaces need not coincide~\cite{Bogachev1998}. In finite dimensions,
\[
    \|v\|_{\Q^{-1}}^2=v^{\mathsf T}\Q^\dagger v,
    \qquad v\in\Ran(B),
\]
so the range formulation and minimum-energy norm have the same
pseudoinverse structure as in the theory of degenerate stationary
diffusions~\cite{DaCostaPavliotis2023}.

Let \((R_t)_{t\geq0}\) be the Ornstein--Uhlenbeck semigroup associated with~\eqref{eq:ou-reference},
\[
    R_t\varphi(x):=\mathbb E[\varphi(Z_t^x)],
\]
and let \((\Lzero,\Dom(\Lzero))\) be its closed generator on \(L^2(\nu)\).
As a concrete sufficient class, let \(\varphi\in C_b^2(\Hh)\) satisfy
the conditions below, where \(\nabla\varphi\) and \(\nabla^2\varphi\) denote the
classical first and second Fr\'echet derivatives of \(\varphi\):
\[
    \nabla\varphi(\Hh)\subset\Dom(A^*),
    \qquad
    A^*\nabla\varphi\in C_b(\Hh;\Hh),
    \qquad
    \Q \nabla^2\varphi\in C_b(\Hh;\mathcal L_1(\Hh)).
\]
Then \(\varphi\in\Dom(\Lzero)\), and its
generator is represented \(\nu\)-almost everywhere by
\begin{equation}\label{eq:L0-expression}
    \Lzero\varphi(x)
    =
    \inner{x}{A^*\nabla\varphi(x)}
    +
    \frac12\Tr\!\left(\Q \nabla^2\varphi(x)\right).
\end{equation}
If \(x\in\Dom(A)\), the first term equals
\(\inner{Ax}{\nabla\varphi(x)}\).

The time-reversal argument also requires the Gaussian reference dynamics to
be reversible.

\begin{assumption}[Reversible Gaussian reference]
\label{ass:OU}
The Ornstein--Uhlenbeck semigroup \((R_t)_{t\geq0}\) is symmetric on \(L^2(\nu)\);
equivalently, \(\Lzero\) is self-adjoint and the reference Ornstein--Uhlenbeck process is
reversible with respect to \(\nu\).
\end{assumption}
\begin{remark}\label{rem:OU-symmetry}
The Ornstein--Uhlenbeck semigroup is symmetric in \(L^2(\nu)\) if and only if
\(\Q h\in\Dom(A)\) and \(A\Q h=\Q A^*h\) for every
\(h\in\Dom(A^*)\); equivalently,
\(S(t)\Q=\Q S^*(t)\) for every \(t\geq0\), see~\cite[Theorem 2.4]{MG02}.
Thus symmetry is strictly stronger than invariance: the Gaussian invariant
measure may exist even when the Ornstein--Uhlenbeck dynamics is not reversible.
\end{remark}

To formulate weak derivatives of densities relative to \(\nu\), we use the
closed Gaussian gradient.

\begin{definition}[Gaussian Sobolev space]
\label{def:gaussian-sobolev}
Let \(\nabla\) denote the closed Gaussian gradient
\[
    \nabla:\Dom(\nabla)\subset L^2(\nu)
    \longrightarrow L^2(\nu;\Hh)
\]
associated with the Gaussian measure \(\nu\). We set
\[
    W^{1,2}(\Hh,\nu):=\Dom(\nabla),
\]
equipped with
\[
    \|u\|_{W^{1,2}(\nu)}^2
    :=\|u\|_{L^2(\nu)}^2
      +\int_{\Hh}\|\nabla u(x)\|_{\Hh}^2\,\nu(\dd x).
\]
All gradients of nonsmooth densities below are understood in this weak
Gaussian Sobolev sense.
\end{definition}

\begin{remark}[Classical and closed gradients]
On the regular class used in~\eqref{eq:L0-expression}, the closed Gaussian
gradient in Definition~\ref{def:gaussian-sobolev} agrees
\(\nu\)-almost everywhere with the classical Fr\'echet derivative. Thus,
when a regular function is viewed as an element of \(W^{1,2}(\Hh,\nu)\),
the two gradients represent the same element of \(L^2(\nu;\Hh)\), and we
use the symbol \(\nabla\) for both. By contrast, \(\nabla^2\varphi\) in
\eqref{eq:L0-expression} is the classical second Fr\'echet derivative; it
does not denote an iteration of the closed gradient.
\end{remark}

The associated integration-by-parts operator is the closed Gaussian
divergence defined next.

\begin{definition}[Gaussian divergence]
\label{def:gaussian-divergence}
The Gaussian divergence is the negative adjoint of the closed gradient.
More precisely, \(Y\in L^2(\nu;\Hh)\) belongs to
\(\Dom(\divnu)\) if there exists \(g\in L^2(\nu)\) such that
\[
    \int_{\Hh}\inner{Y}{\nabla\varphi}\,\dd\nu
    =
    -\int_{\Hh}\varphi g\,\dd\nu,
    \qquad \varphi\in W^{1,2}(\Hh,\nu).
\]
The function \(g\), which is unique in \(L^2(\nu)\), is denoted by
\(\divnu Y\).
\end{definition}

The closed gradient and its adjoint divergence are standard objects of
Gaussian Sobolev analysis; see~\cite[Chapter~5]{Bogachev1998}
and~\cite{Bogachev2018}. In finite dimensions, if
\(\nu(\dd x)=\gamma(x)\,\dd x\) and \(Y\) is smooth, then
\[
    \divnu Y
    =\operatorname{div}Y+\inner{Y}{\nabla\log\gamma},
\]
which explains how \(\divnu\) replaces the Euclidean divergence when the
reference measure is Gaussian.

Under Assumption~\ref{ass:OU}, symmetry of the Ornstein--Uhlenbeck
semigroup identifies the closed generator with the Gaussian divergence
form. More precisely, for every
\(u\in\Dom(\Lzero)\cap W^{1,2}(\Hh,\nu)\), the closed Dirichlet-form
identity \cite[Proposition~2.47]{DP04}
\[
    \int_{\Hh}\varphi\,\Lzero u\,\dd\nu
    =-\frac12\int_{\Hh}\inner{\Q \nabla u}{\nabla\varphi}\,\dd\nu,
    \qquad \varphi\in W^{1,2}(\Hh,\nu),
\]
and Definition~\ref{def:gaussian-divergence} give
\[
    \Q \nabla u\in\Dom(\divnu),
    \qquad
    \divnu(\Q \nabla u)=2\Lzero u
    \quad\text{in }L^2(\nu).
\]
See~\cite{Bogachev2018, DP04} for this closed Dirichlet-form realization of
the symmetric Ornstein--Uhlenbeck generator.

We now relate the stationary law of the nonlinear equation to the Gaussian
reference measure.

\begin{assumption}[Density relative to the Gaussian reference]
\label{ass:density}
A probability measure \(\mu\) on \(\Hh\) is absolutely continuous with
respect to \(\nu\):
\[
    \mu(\dd x)=\rho(x)\nu(\dd x),
    \qquad \rho>0\quad\nu\text{-a.e.},
    \qquad \int_{\Hh}\rho\,\dd\nu=1.
\]
\end{assumption}

The following regularity conditions justify the Girsanov transform, the
Hilbert-space It\^o formula, and the pathwise reversal calculation.

\begin{assumption}[Drift, logarithmic density, and strong realization]
\label{ass:path-regularity}
The following conditions hold.
\begin{enumerate}[label=\textup{(\roman*)}]
\item
There exists \(\theta\in C_b^2(\Hh;\Hh)\) such that
\begin{equation}\label{eq:F-Qtheta}
    F=\Q\theta,
    \qquad
    \theta(\Hh)\subset\Dom(A^*),
    \qquad
    A^*\theta\in C_b(\Hh;\Hh).
\end{equation}

\item
The density in Assumption~\ref{ass:density} has a strictly positive Borel
representative, still denoted by \(\rho\), such that
\(\log\rho\in C_b^2(\Hh)\) and
\[
    (\nabla\log\rho)(\Hh)\subset\Dom(A^*),
    \qquad
    A^*\nabla\log\rho\in C_b(\Hh;\Hh).
\]
We use this representative throughout.

\item
Equation~\eqref{eq:intro-spde} admits a stationary realization \(X\), adapted
to a filtration \((\mathcal F_t)_{t\geq0}\) with respect to which \(W\) is a
cylindrical Wiener process, such that \(X_t\sim\mu\) for every \(t\geq0\).
For each \(T>0\),
\[
\begin{aligned}
    &X\in C([0,T];\Hh)\quad\text{a.s.},\\
    &X_r\in\Dom(A)
      \quad\text{for }\dd r\otimes\Pp\text{-a.e. }(r,\omega),
      \qquad
      \int_0^T\|AX_r\|_{\Hh}\,\dd r<\infty
      \quad\text{a.s.},
\end{aligned}
\]
and
\begin{equation}\label{eq:strong-solution-form}
    X_t
    =
    X_0+\int_0^t\bigl(AX_r+F(X_r)\bigr)\,\dd r
        +\int_0^t B\,\dd W_r,
    \qquad 0\leq t\leq T,
\end{equation}
holds in \(\Hh\) for all \(t\in[0,T]\) on a single event of probability one.
This is the meaning of ``strong solution'' used below.
\end{enumerate}
\end{assumption}

\begin{remark}[Density representative and its logarithm]
\label{rem:density-representative}
As the Radon--Nikodym derivative \(\dd\mu/\dd\nu\), the density \(\rho\) in
Assumption~\ref{ass:density} is initially defined only up to
\(\nu\)-almost-everywhere equality, that is, as an equivalence class in
\(L^1(\nu)\). Pointwise expressions involving \(\rho\) therefore require a
chosen representative. Assumption~\ref{ass:path-regularity}~(ii) fixes a
strictly positive Borel representative; throughout the paper,
\(\log(\rho)\) denotes its pointwise logarithm, with no separate symbol
introduced. This notation is used because density derivatives and
time-reversal terms take the natural forms
\(\nabla\rho=\rho \nabla\log\rho\) and
\(\log\rho(X_t)-\log\rho(X_s)\); the corresponding
closed-generator identities are established in
Lemma~\ref{lem:density-regularity}. Although two representatives agree at
each fixed time along the stationary process because \(X_t\sim\mu\ll\nu\),
fixing one Borel representative is necessary for the pathwise evaluations
and It\^o formulas used below.
\end{remark}

\begin{remark}[Girsanov admissibility]
\label{rem:girsanov-admissibility}
For the stationary realization in Assumption~\ref{ass:path-regularity}, set
\(G:=B^*\theta\). Then
\[
    F=\Q\theta=BG,
    \qquad
    G(\Hh)\subset(\ker B)^\perp,
\]
and \(G\) is bounded and globally Lipschitz. Likewise, \(F=\Q\theta\) is
bounded and globally Lipschitz, so standard semilinear well-posedness yields
a unique mild solution for every initial state and a Markov transition
semigroup; the realization above is the corresponding stationary Markov
process~\cite{DZ92,DZ96}. Since \(X\) is adapted with
continuous paths, \(G(X)\) is progressively measurable. Moreover, for every
finite interval \([s,t]\),
\[
    \E\exp\left\{
        \frac12\int_s^t\|G(X_r)\|_{\Uh}^2\,\dd r
    \right\}
    \leq
    \exp\left\{\frac{t-s}{2}\|G\|_\infty^2\right\}<\infty.
\]
Consequently, the noise-range, regularity, and Novikov conditions in
Lemma~\ref{lem:girsanov} hold automatically on every finite time interval;
no bounded inverse of \(B\) is required. In fact,
\(G=B^\dagger F\) is the unique minimal-norm control representing \(F\),
and
\[
    \|F(x)\|_{\Q^{-1}}=\|G(x)\|_{\Uh},
    \qquad x\in\Hh.
\]
Similarly, with
\[
    \Gamma:=B^*(2\theta-\nabla\log\rho),
    \qquad
    \widehat V:=2F-\Q \nabla\log\rho=B\Gamma,
\]
one has \(\Gamma\in C_b(\Hh;(\ker B)^\perp)\) and
\[
    \|\widehat V(x)\|_{\Q^{-1}}=\|\Gamma(x)\|_{\Uh},
    \qquad
    \widehat V\in L^2(\mu;\Hh_{\Q}).
\]
For degenerate finite-dimensional stationary diffusions, failure of the
corresponding noise-range condition can lead to mutually singular path laws
and infinite entropy production~\cite{DaCostaPavliotis2023}. This suggests an
analogous infinite-dimensional singular regime, which is not covered here.
\end{remark}

The next lemma turns the preceding pointwise regularity into the closed
generator and stationary Fokker--Planck identities used throughout the
proofs.

\begin{lemma}[Density regularity and stationary equation]
\label{lem:density-regularity}
Under Assumptions~\ref{ass:AFQ},~\ref{ass:OU},~\ref{ass:density}, and
\ref{ass:path-regularity}, both \(\log\rho\) and \(\rho\) belong to
\(\Dom(\Lzero)\cap W^{1,2}(\Hh,\nu)\). Moreover, with \(\Lzero\) denotes the generator of the Ornstein--Uhlenbeck process \eqref{eq:ou-reference} from \eqref{eq:L0-expression},
\begin{equation}\label{eq:L0-rho-chain-rule}
    \Lzero\rho
    =\rho\left[
        \Lzero(\log\rho)+\frac12\|B^*\nabla\log\rho\|_{\Uh}^2
      \right]
    \quad\text{in }L^2(\nu),
\end{equation}
and
\begin{equation}\label{eq:log-gradient-consistency}
    \nabla\log\rho=\frac{\nabla\rho}{\rho},
    \qquad \nu\text{-a.e.}.
\end{equation}
For \(0\leq s\leq t<\infty\), the stationary strong realization satisfies
\begin{equation}\label{eq:logrho-ito}
\begin{aligned}
    \log\rho(X_t)-\log\rho(X_s)
    &=\int_s^t
      \langle B^*\nabla\log\rho(X_r),\dd W_r\rangle_{\Uh} \\
    &\quad+\int_s^t
      \bigl[\Lzero(\log\rho)+\inner{F}{\nabla\log\rho}\bigr](X_r)\,\dd r
\end{aligned}
\end{equation}
almost surely. Finally,
\[
    \rho F\in\Dom(\divnu),
\]
and the stationary Fokker--Planck equation holds:
\begin{equation}\label{eq:FP}
    \Lzero\rho-\divnu(\rho F)=0
    \qquad\text{in }L^2(\nu).
\end{equation}
\end{lemma}

\begin{proof}
Since \(\log\rho\) and \(\rho\), together with
their Fr\'echet derivatives, are bounded, both functions belong to
\(W^{1,2}(\Hh,\nu)\); the Gaussian Sobolev chain rule gives
\eqref{eq:log-gradient-consistency}.
Since \(\Q\in\mathcal L_1(\Hh)\) and \(\nabla^2(\log\rho)\) is bounded and
continuous, \(\Q \nabla^2(\log\rho)\in
C_b(\Hh;\mathcal L_1(\Hh))\). Hence~\eqref{eq:L0-expression} gives
\[
    [\Lzero(\log\rho)](x)
    =\inner{x}{A^*\nabla(\log\rho)(x)}
     +\frac12\Tr\!\left(\Q \nabla^2(\log\rho)(x)\right)
\]
for \(\nu\)-almost every \(x\). We use the Borel function on the right-hand
side as the representative of \(\Lzero(\log\rho)\). The Fr\'echet chain rule
gives
\[
    \nabla\rho=\rho \nabla\log\rho,
    \qquad
    \nabla^2\rho=\rho\bigl(\nabla^2(\log\rho)
      +\nabla\log\rho\otimes \nabla\log\rho\bigr),
    \qquad
    A^*\nabla\rho=\rho A^*\nabla\log\rho.
\]
Thus \(\rho\) satisfies the same generator criterion. Applying
\eqref{eq:L0-expression} and using
\(\Tr(\Q(u\otimes u))=\|B^*u\|_{\Uh}^2\) proves
\eqref{eq:L0-rho-chain-rule}. 

For \(x\in\Dom(A)\), operator duality and trace cyclicity give
\[
    \inner{Ax}{\nabla\log\rho(x)}=\inner{x}{A^*\nabla\log\rho(x)},
    \qquad
    \Tr_{\Uh}\!\left(B^*\nabla^2(\log\rho)(x)B\right)
    =\Tr_{\Hh}\!\left(\Q \nabla^2(\log\rho)(x)\right).
\]
The Hilbert-space It\^o formula applied to
\eqref{eq:strong-solution-form} therefore yields~\eqref{eq:logrho-ito};
see~\cite[Chapter~4]{DZ92}.
Its stochastic integral is square integrable because \(B^*\nabla\log\rho\) is
bounded; its drift is integrable by stationarity because \(\rho\) is
bounded, \(\Lzero(\log\rho)\in L^2(\nu)\), and \(F,\nabla\log\rho\) are bounded.

It remains to prove the stationary equation \eqref{eq:FP}. Let \(\varphi\) be a
bounded \(C^2\)-cylindrical function, with bounded derivatives, generated
by directions in \(\Dom(A^*)\). Its It\^o drift is integrable: the identity
\(\inner{AX_r}{\nabla\varphi(X_r)}=\inner{X_r}{A^*\nabla\varphi(X_r)}\), the boundedness
of \(A^*\nabla\varphi\), and \(X_r\sim\rho\nu\) give integrability of the linear
term, while the remaining terms are bounded. Applying It\^o's formula to
\(\varphi(X_t)\), taking expectations, and using stationarity therefore gives
\[
    \int_{\Hh}
    \bigl[\Lzero\varphi+\inner{F}{\nabla\varphi}\bigr]\rho\,\dd\nu=0.
\]
Since \(\Lzero\) is self-adjoint and \(\rho\in\Dom(\Lzero)\),
\begin{equation}\label{eq:FP-weak}
    \int_{\Hh}\inner{\rho F}{\nabla\varphi}\,\dd\nu
    =-\int_{\Hh}\varphi\,\Lzero\rho\,\dd\nu.
\end{equation}
By cylindrical approximation (with directions in the dense subspace
\(\Dom(A^*)\)) and scalar truncation, such test functions are dense in
\(W^{1,2}(\Hh,\nu)\); see~\cite[Chapter~5]{Bogachev1998}. Both sides \eqref{eq:FP-weak} are
continuous in the
\(W^{1,2}\)-norm because
\(\rho F\in L^2(\nu;\Hh)\) and \(\Lzero\rho\in L^2(\nu)\). The identity
therefore extends to every \(\varphi\in W^{1,2}(\Hh,\nu)\). By
Definition~\ref{def:gaussian-divergence},
\(\rho F\in\Dom(\divnu)\) and
\(\divnu(\rho F)=\Lzero\rho\), which proves~\eqref{eq:FP}.
\end{proof}

These assumptions are stronger than those required for mild well-posedness,
but permit the Hilbert-space It\^o formula and the direct
strong-semimartingale time-reversal calculation; see~\cite{DZ92}. In
particular, \(F=\Q\theta\) provides the noise-range factorization and trace
regularity used below. Since \(\Q\) is trace class and \(\nabla\theta\) is
bounded and continuous, \(\Q \nabla\theta\) is trace class with uniformly
bounded, continuous trace. Compare the distinct Onsager--Machlup small-tube
setting in~\cite{BardinaRoviraTindel2003}.

The next proposition gives a directly verifiable diagonal criterion for the
strong-solution part of the framework. Its stochastic-convolution estimate is
a strengthened form of the standard regularity criteria for linear stochastic
evolution equations~\cite{DZ92}.

\begin{proposition}[Diagonal strong-solution criterion]
\label{prop:diagonal-strong-framework}
Let \(\Uh=\Hh\), and suppose that \(A\) is self-adjoint and that \(A\) and
\(B\) are diagonal with respect to an orthonormal basis
\(\{e_k\}_{k\geq1}\):
\[
    Ae_k=-\lambda_ke_k,
    \qquad
    Be_k=b_ke_k.
\]
Assume that \(\lambda_k\geq\lambda_*>0\), \(b_k>0\), and
\begin{equation}\label{eq:diagonal-strong-summability}
    \sum_{k=1}^{\infty}\lambda_kb_k^2<\infty.
\end{equation}
Then \(B\) is Hilbert--Schmidt, \(\Q=B^2\) is injective and trace class,
and the Ornstein--Uhlenbeck equation \eqref{eq:ou-reference} has the reversible invariant measure
\(\nu=\mathcal N(0,Q_\infty)\), where
\begin{equation}\label{invariant-covariance}
    Q_\infty e_k=\frac{b_k^2}{2\lambda_k}e_k.
\end{equation}
Its stationary realization is a continuous strong solution, and
\[
    \int_{\Hh}\|Ax\|_{\Hh}^2\,\nu(\dd x)
    =
    \frac12\sum_{k=1}^{\infty}\lambda_kb_k^2<\infty.
\]

If \(\theta\in C_b^2(\Hh;\Hh)\) and \(F=\Q\theta\), then the corresponding
nonlinear equation \eqref{eq:intro-spde} has a unique mild solution for every initial state. The
solution is strong on \([0,T]\) for \(x\in\Dom(A)\), and on
\([\varepsilon,T]\) for arbitrary \(x\in\Hh\) and \(\varepsilon>0\).
Consequently, every invariant measure is supported on \(\Dom(A)\), and every
stationary realization is a continuous strong solution.
\end{proposition}

\begin{proof}
The summability condition gives
\[
    \|B\|_{\mathcal L_2(\Hh,\Hh)}^2
    =\sum_{k=1}^{\infty}b_k^2
    \leq
    \frac1{\lambda_*}
    \sum_{k=1}^{\infty}\lambda_kb_k^2<\infty.
\]
Thus \(B\) is Hilbert--Schmidt; moreover, \(b_k>0\) and
\(\Tr(\Q)=\sum_kb_k^2<\infty\) show that \(\Q=B^2\) is injective and trace
class. For the stochastic convolution
\[
    W_A(t):=\int_0^tS(t-r)B\,\dd W_r,
\]
It\^o's isometry gives
\[
    \E\|AW_A(t)\|_{\Hh}^2
    =
    \frac12\sum_{k=1}^{\infty}
    \lambda_kb_k^2\bigl(1-e^{-2\lambda_kt}\bigr),
\]
and consequently
\[
    \E\int_0^T\|AW_A(t)\|_{\Hh}^2\,\dd t
    \leq
    \frac T2\sum_{k=1}^{\infty}\lambda_kb_k^2<\infty.
\]
Thus \(W_A(t)\in\Dom(A)\) almost surely for every fixed \(t>0\), and
\(AW_A\in L^2(\Omega\times(0,T);\Hh)\). Together with the
\(\Hh\)-continuity of \(W_A\), this yields a continuous strong Ornstein--Uhlenbeck solution
for initial data in \(\Dom(A)\); see~\cite{DZ92} for the corresponding
stochastic-convolution regularity criteria.
Integrating each spectral mode gives the stated formula for \(Q_\infty\) and
\[
    \Tr(Q_\infty)
    =
    \frac12\sum_{k=1}^{\infty}\frac{b_k^2}{\lambda_k}
    \leq
    \frac1{2\lambda_*}\sum_{k=1}^{\infty}b_k^2<\infty.
\]
Since \(A\) and \(\Q\) are diagonal in the same basis,
\(S(t)\Q=\Q S(t)\); hence the Ornstein--Uhlenbeck semigroup is symmetric in \(L^2(\nu)\).
The same spectral calculation gives the displayed \(A\)-moment identity, so
the stationary Ornstein--Uhlenbeck realization is strong.

For the nonlinear equation \eqref{eq:intro-spde}, we observe that
\[
    A\Q e_k=-\lambda_kb_k^2e_k,
    \qquad
    \|A\Q\|_{\mathcal L_1(\Hh)}
    =\sum_{k=1}^{\infty}\lambda_kb_k^2<\infty.
\]
Hence \(\Q:\Hh\to\Dom(A)\) is continuous in the graph norm, and both
\(F=\Q\theta\) and \(AF=A\Q\theta\) are bounded and globally Lipschitz.
The standard fixed-point argument gives the unique mild solution. If
\(x\in\Dom(A)\), the variation-of-constants formula, the stochastic
convolution estimate above, and closedness of \(A\) yield the strong integral
equation on \([0,T]\). Analytic smoothing gives the same conclusion on
\([\varepsilon,T]\) for arbitrary \(x\in\Hh\); see~\cite{DZ92,EngelNagel2000}.
In particular,
\[
    P_t(x,\Dom(A))=1,
    \qquad t>0.
\]
If \(\mu\) is invariant, then
\[
    \mu(\Dom(A))
    =
    \int_{\Hh}P_t(x,\Dom(A))\,\mu(\dd x)=1,
\]
which proves the assertion about stationary strong solutions.
\end{proof}

The weaker condition
\[
    \sum_{k=1}^{\infty}\frac{b_k^2}{1+\lambda_k}<\infty,
\]
used for mild solutions in~\cite{BardinaRoviraTindel2003}, does not imply
that \(B\) is Hilbert--Schmidt and is therefore insufficient for the
strong-semimartingale proof of
Proposition~\ref{prop:time-reversal-girsanov}.
Condition~\eqref{eq:diagonal-strong-summability} is a convenient verifiable
sufficient condition, not an additional standing restriction on the
abstract framework. Although \(b_k>0\) makes \(B\) injective with dense
range, \(B\) is compact and has no bounded inverse in infinite dimensions;
the notation \(B^\dagger\) must therefore be retained.

\subsection{Entropy production and reversibility}
Let \(X\) be the solution of the nonlinear equation~\eqref{eq:intro-spde},
started from its invariant measure
\[
    X_0\sim\mu=\rho\nu.
\]
Here \(\nu\) is the reversible invariant Gaussian measure of the reference
Ornstein--Uhlenbeck equation~\eqref{eq:ou-reference}. Thus \(X\) is
stationary and \(\mathcal L(X_t)=\mu\) for every \(t\geq0\).

For \(0\le s<t<\infty\), let
\[
    \Pp^+_{[s,t]}
    :=
    \mathcal L_\mu\bigl((X_r)_{s\leq r\leq t}\bigr)
\]
be the forward path law of~\eqref{eq:intro-spde}. The time-reversal map
\[
    \Theta_{s,t}:C([s,t];\Hh)\to C([s,t];\Hh)
\]
is defined by
\[
    (\Theta_{s,t}\omega)(r)=\omega(s+t-r).
\]
The corresponding reversed path law is
\[
    \Pp^-_{[s,t]}
    :=
    \Pp^+_{[s,t]}\circ\Theta_{s,t}^{-1}.
\]
For probability measures \(P\) and \(R\) on the same path space, we write the relative entropy of \(P\) with respect to \(R\) as
\[
    \mathcal H(P\,|\,R)
    :=
    \int_{\Omega}\log\!\left(\frac{\dd P}{\dd R}\right)\dd P
\]
when \(P\ll R\), and set \(\mathcal H(P\,|\,R)=+\infty\) otherwise.
Entropy production is the short-time relative-entropy rate of the forward
path law with respect to its time reversal.
This path-space viewpoint is standard for finite-dimensional diffusions and
nonequilibrium stochastic dynamics~\cite{LebowitzSpohn1999, QW99}; for
stationary diffusions with a possibly singular diffusion matrix,
see~\cite{DaCostaPavliotis2023}. Infinite-dimensional diffusion reversal has
also been studied under local entropy or Malliavin-type
hypotheses~\cite{FollmerWakolbinger1986,MilletNualartSanz1989}. Here the reversed law is
instead compared through a reversible Gaussian Ornstein--Uhlenbeck reference.

Throughout, conditional expectations given \(X_t=x\) are taken in a fixed
Borel Markov version. Since \(X_t\sim\mu\), these conditional quantities,
including the local entropy production density, are intrinsically determined
only for \(\mu\)-almost every \(x\).

\begin{definition}[Instantaneous entropy production density]
The instantaneous entropy production density at \((t,x)\) is
\begin{equation}
\label{def:local-entropy-production-density}
    \operatorname{ep}(t,x)
    :=
    \lim_{\Delta t\downarrow0}
    \frac{1}{\Delta t}
    \mathbb E_{\Pp^+_{[t,t+\Delta t]}}
    \left[
        \log
        \left(\frac{\dd\Pp^+_{[t,t+\Delta t]}}
             {\dd\Pp^-_{[t,t+\Delta t]}}\right)
        \,\middle|\, X_t=x
    \right],
\end{equation}
provided that
\(\Pp^+_{[t,t+\Delta t]}\ll\Pp^-_{[t,t+\Delta t]}\) for sufficiently small
\(\Delta t\) and that the conditional limit exists.
\end{definition}

The corresponding unconditional short-time relative entropy defines the
total entropy production rate.

\begin{definition}[Instantaneous entropy production rate]
The instantaneous entropy production rate at time \(t\) is defined by
\begin{equation}
\label{def:instantaneous-entropy-production}
    \operatorname{EP}(t)
    :=
    \lim_{\Delta t\downarrow0}
    \frac{1}{\Delta t}
    \mathcal H\!\left(
        \Pp^+_{[t,t+\Delta t]}
        \,\middle|\,
        \Pp^-_{[t,t+\Delta t]}
    \right),
\end{equation}
whenever the limit exists.
\end{definition}

With these notions in place, the entropy production formula can be stated
explicitly.

\begin{theorem}[Entropy production formula for SPDEs]
\label{thm:main}
Suppose that the reference Assumptions~\ref{ass:AFQ} and~\ref{ass:OU}, the
density Assumption~\ref{ass:density}, and the path-regularity
Assumption~\ref{ass:path-regularity} hold. Define the irreversibility field by
\begin{equation}\label{eq:irreversible-drift}
    V(x):=2F(x)-\Q \nabla\log\rho(x).
\end{equation}
Then, for every \(t\geq0\), the entropy production density is
\begin{equation}
\label{eq:ep-density-main}
    \ep(t,x)=\ep(x)
    =
    \frac12\|V(x)\|_{\Q^{-1}}^2
    \qquad\text{for }\mu\text{-a.e. }x.
\end{equation}
The entropy production rate is
\begin{equation}
\label{eq:EP-stationary-main}
    \EP(t)=\EP
    =
    \int_{\Hh}\ep(x)\,\mu(\dd x)
    =
    \frac12
    \int_{\Hh}
    \|V(x)\|_{\Q^{-1}}^2\rho(x)
    \nu(\dd x).
\end{equation}
\end{theorem}

\begin{remark}[Finite-dimensional entropy production formula]
Let
\(\Hh=\mathbb R^d\), write
\(\nu(\dd x)=\gamma(x)\,\dd x\) and
\(\mu(\dd x)=p(x)\,\dd x\), and set
\(p=\rho\gamma\) and \(b(x)=Ax+F(x)\). Under the standing full-support and
reversible-reference assumptions, \(\Q\) is positive definite: indeed,
\(AQ_\infty=-\Q/2\), while \(A\) and \(Q_\infty\) are invertible in finite
dimensions. Moreover,
\[
    2Ax=\Q\nabla\log\gamma(x).
\]
Consequently,
\[
    V=2F-\Q\nabla\log\rho
     =2b-\Q\nabla\log p,
\]
and~\eqref{eq:EP-stationary-main} becomes
\[
    \EP
    =\frac12\int_{\mathbb R^d}
      (2b-\Q\nabla\log p)^{\mathsf T}\Q^{-1}
      (2b-\Q\nabla\log p)\,p\,\dd x.
\]
This is the diffusion form of the entropy-production
formula~\cite[Theorem 4.1.7]{jiang2004}. For a singular finite-dimensional diffusion matrix,
analogous results require an explicit range condition and a Moore--Penrose
inverse~\cite{DaCostaPavliotis2023}; that case is structurally parallel to, but not
directly covered by, the present finite-dimensional standing assumptions.
\end{remark}

Theorem~\ref{thm:main} shows that entropy production is nonnegative and that
\(\EP=0\) precisely when \(V=0\) \(\mu\)-a.e. We next relate this condition
to the generator, probability current, detailed balance, and pathwise
reversibility.

The Markov semigroup of~\eqref{eq:intro-spde},
\[
    P_t f(x):=\E[f(X_t^x)]
\]
extends to a contraction semigroup on \(L^2(\mu)\). Let
\((L,\Dom(L))\) be its closed generator. It is the nonlinear perturbation of
the Ornstein--Uhlenbeck generator associated with~\eqref{eq:ou-reference}; on the natural
common domain,
\[
    L\varphi
    =
    \Lzero\varphi+\inner{F}{\nabla\varphi}
\]
whenever the right-hand side belongs to \(L^2(\mu)\); see~\cite{DZ92,DZ96}
for the Markov-semigroup framework for semilinear stochastic evolution
equations.

We first recall the semigroup formulation of reversibility.

\begin{definition}[Detailed balance]
\label{def:detailed-balance}
The stationary solution of~\eqref{eq:intro-spde} satisfies detailed balance
with respect to \(\mu\) if
\[
    \int_{\Hh} f\,P_tg\,\dd\mu
    =
    \int_{\Hh} g\,P_tf\,\dd\mu
\]
for all bounded measurable functions \(f,g\) and all \(t\ge0\). Equivalently,
the semigroup \((P_t)_{t\ge0}\) is self-adjoint in \(L^2(\mu)\).
\end{definition}

The corresponding path-space notion compares the stationary law with its
time reversal.

\begin{definition}[Pathwise reversibility]
\label{def:path-reversibility}
The stationary solution \(X\) of~\eqref{eq:intro-spde} is pathwise
reversible if, for every \(T>0\),
\[
    \mathcal L(X_t,\ 0\le t\le T)
    =
    \mathcal L(X_{T-t},\ 0\le t\le T).
\]
Equivalently,
\[
    \Pp^+_{[0,T]}=\Pp^-_{[0,T]}.
\]
\end{definition}

To express stationary transport relative to \(\nu\), we introduce the
Gaussian-reference probability current.

\begin{definition}[Stationary probability current]
\label{def:stationary-current}
The stationary probability current relative to the reversible
Ornstein--Uhlenbeck reference~\eqref{eq:ou-reference} is
\[
    J_\mu
    :=
    \rho F-\frac12\Q \nabla\rho.
\]
\end{definition}

Under Lemma~\ref{lem:density-regularity}, we have
\[
    \nabla\rho=\rho \nabla\log\rho,
\]
we may also write
\[
    J_\mu
    =
    \frac{\rho}{2}
    \left(
        2F-\Q \nabla\log\rho
    \right).
\]
Then, \(J_\mu\) is defined \(\mu\)-almost
everywhere as an \(\Hh_{\Q}\)-valued field and belongs to
\(\Dom(\divnu)\). For the reversible Ornstein--Uhlenbeck generator,
\[
    \Lzero\rho=\frac12\divnu(\Q \nabla\rho).
\]
Combining this identity with~\eqref{eq:FP} gives
\[
    \divnu J_\mu=0
\]
in \(L^2(\nu)\).

With the finite-dimensional notation used after Theorem~\ref{thm:main}, the
usual Lebesgue-density current is
\[
    j_p:=pb-\frac12\Q\nabla p=\gamma J_\mu.
\]
Thus \(J_\mu\) is the Gaussian-reference analogue of the usual stationary
current~\cite{DaCostaPavliotis2023, QW99}; the reversible linear drift \(Ax\) has
already been absorbed into the reference measure \(\nu\).

The following theorem collects the equivalent analytic, probabilistic, and
pathwise characterizations of reversibility.

\begin{theorem}[Equivalent criteria for reversibility]
\label{thm:reversibility-criteria}
In the stationary setting described above, suppose that the hypotheses of
Theorem~\ref{thm:main} hold, and let \(V\) be the irreversibility field defined
in~\eqref{eq:irreversible-drift}.
Then the following statements are equivalent:

\begin{enumerate}[label=\textup{(\roman*)}]
    \item The entropy production vanishes:
    \[
        \EP=0.
    \]

    \item The irreversibility field vanishes:
    \[
        V(x)=0,
        \qquad \mu\text{-a.e..}
    \]

    \item The stationary probability current vanishes:
    \[
        J_\mu=0,
        \qquad \mu\text{-a.e..}
    \]

    \item The nonlinear drift satisfies the gradient condition
    \[
        F=\frac12\Q \nabla\log\rho,
        \qquad \mu\text{-a.e..}
    \]

    \item The closed generator \(L\) is self-adjoint in \(L^2(\mu)\).

    \item The semigroup satisfies detailed balance:
    \[
        \int_{\Hh} f\,P_tg\,\dd\mu
        =
        \int_{\Hh} g\,P_tf\,\dd\mu,
        \qquad t\ge0.
    \]

    \item The stationary solution is pathwise reversible:
    \[
        \Pp^+_{[0,T]}=\Pp^-_{[0,T]},
        \qquad T>0.
    \]
\end{enumerate}

With
\[
    \Phi:=-\log\rho\in C_b^2(\Hh),
    \qquad
    Z_\Phi:=\int_{\Hh}e^{-\Phi(x)}\,\nu(\dd x)=1,
    \qquad
    \rho=Z_\Phi^{-1}e^{-\Phi},
\]
the above conditions are also equivalent to the gradient structure
\[
    F=-\frac12\Q \nabla\Phi,
    \qquad \mu\text{-a.e..}
\]
\end{theorem}

The equivalence of semigroup symmetry, self-adjointness of the generator,
detailed balance, and stationary path reversal belongs to the general
theory of symmetric Markov processes~\cite[Theorem 4.3.3]{jiang2004}. The
additional content here is the equivalence with the explicit conditions
\(V=0\), \(J_\mu=0\), and vanishing entropy production.

\subsection{A reversible diagonal strong-solution example}

The formal choice \(\Q=I\) is not compatible with the strong-solution
framework used in Proposition~\ref{prop:time-reversal-girsanov}: in an
infinite-dimensional state space, \(I\) is not trace class and the
corresponding cylindrical noise is not an \(\Hh\)-valued semimartingale.
This is the standard distinction between cylindrical and Hilbert-space-valued
Wiener noise~\cite{DZ92}.
We therefore replace that choice by the following injective trace-class
diagonal covariance, while retaining an infinite-dimensional noise in every
mode.

Let \(\Hh=\Uh\) have orthonormal basis \(\{e_k\}_{k\geq1}\), and let
\(A:\Dom(A)\subset\Hh\to\Hh\) be self-adjoint with
\[
    Ae_k=-\lambda_ke_k,
    \qquad
    0<\lambda_1\leq\lambda_2\leq\cdots,
    \qquad
    \lambda_k\longrightarrow\infty.
\]
Assume
\begin{equation}\label{eq:example-spectral-summability}
    \sum_{k=1}^{\infty}\frac1{\lambda_k}<\infty.
\end{equation}
Then \(S(t)=e^{tA}\) is exponentially stable. Define
\[
    b_k:=\frac1{1+\lambda_k},
    \qquad
    q_k:=b_k^2,
    \qquad
    Be_k:=b_ke_k,
    \qquad
    \Q e_k=q_ke_k.
\]
Since
\[
    \sum_{k=1}^{\infty}\lambda_kb_k^2
    =
    \sum_{k=1}^{\infty}
    \frac{\lambda_k}{(1+\lambda_k)^2}<\infty,
\]
Proposition~\ref{prop:diagonal-strong-framework} applies. In particular,
\(B\in\mathcal L_2(\Hh,\Hh)\), and the stationary stochastic convolution
belongs to \(\Dom(A)\) almost surely at each fixed time, with its \(A\)-image
square-integrable on finite time intervals. The Ornstein--Uhlenbeck reference equation \eqref{eq:ou-reference} is
\begin{equation}\label{eq:example-OU}
    \dd Z_t=AZ_t\,\dd t+B\,\dd W_t,
\end{equation}
and its invariant covariance in \eqref{invariant-covariance} satisfies
\[
    Q_\infty e_k=\frac{q_k}{2\lambda_k}e_k.
\]
Thus \(Q_\infty\) is injective and trace class. Since \(A\) and \(\Q\) are
self-adjoint and diagonal in the same basis, the Ornstein--Uhlenbeck semigroup is symmetric
in \(L^2(\nu)\), where
\[
    \nu=\N(0,Q_\infty).
\]

For \(x\in\Hh\), write \(x_k:=\inner{x}{e_k}\), and set
\[
    a_k:=\frac1{(1+\lambda_k)^2}.
\]
The series
\[
    \Phi(x):=2\sum_{k=1}^{\infty}a_k\cos(x_k)
\]
converges uniformly and defines a bounded \(C^2\) potential. Define
\[
    \theta(x)
    :=
    -\frac12\nabla\Phi(x)
    =
    \sum_{k=1}^{\infty}a_k\sin(x_k)e_k
\]
and
\begin{equation}\label{eq:example-drift}
\begin{aligned}
    F(x)
    &:=
    \Q\theta(x)
    =
    \sum_{k=1}^{\infty}q_ka_k\sin(x_k)e_k, \\
    G(x)
    &:=
    B^*\theta(x)
    =
    \sum_{k=1}^{\infty}\sqrt{q_k}\,a_k\sin(x_k)e_k.
\end{aligned}
\end{equation}
Then \(F=BG\), and the nonlinear equation becomes
\begin{equation}\label{eq:example-nonlinear-equation}
    \dd X_t
    =
    \left(
        AX_t+
        \sum_{k=1}^{\infty}q_ka_k
        \sin\bigl(\inner{X_t}{e_k}\bigr)e_k
    \right)\dd t
    +B\,\dd W_t.
\end{equation}

The maps \(F\) and \(G\) are bounded and globally Lipschitz. Moreover,
\[
    \nabla\theta(x)h
    =
    \sum_{k=1}^{\infty}
    a_k\cos(x_k)\inner{h}{e_k}e_k,
\]
and
\[
    \nabla^2\theta(x)[h_1,h_2]
    =-\sum_{k=1}^{\infty}
      a_k\sin(x_k)
      \inner{h_1}{e_k}\inner{h_2}{e_k}e_k.
\]
The bounds
\[
    \|\nabla\theta(x)\|_{\mathcal L(\Hh)}
    \leq \sup_k a_k,
    \qquad
    \sup_{\|h_1\|_{\Hh},\|h_2\|_{\Hh}\leq1}
    \|\nabla^2\theta(x)[h_1,h_2]\|_{\Hh}
    \leq \sup_k a_k
\]
and the coordinatewise Lipschitz estimates show that
\(\theta\in C_b^2(\Hh;\Hh)\). Furthermore,
\[
\begin{aligned}
    \Q \nabla\theta(x)h
    &=
    \sum_{k=1}^{\infty}
    q_ka_k\cos(x_k)\inner{h}{e_k}e_k, \\
    \sup_{x\in\Hh}
    \|\Q \nabla\theta(x)\|_{\mathcal L_1(\Hh)}
    &\leq
    \sum_{k=1}^{\infty}q_ka_k<\infty.
\end{aligned}
\]
The map \(x\mapsto\Q \nabla\theta(x)\) is trace-norm continuous. Moreover,
\[
\begin{aligned}
    A^*\theta(x)=A\theta(x)
    &=-\sum_{k=1}^{\infty}\lambda_ka_k\sin(x_k)e_k, \\
    \sup_{x\in\Hh}\|A^*\theta(x)\|_{\Hh}^2
    &\leq
    \sum_{k=1}^{\infty}\lambda_k^2a_k^2<\infty.
\end{aligned}
\]
Hence the requirements on \(\theta\) in
Assumption~\ref{ass:path-regularity} are satisfied.

The stationary Gaussian law obeys
\[
    \int_{\Hh}\|Ax\|_{\Hh}^2\,\nu(\dd x)
    =
    \frac12\sum_{k=1}^{\infty}\lambda_kq_k<\infty.
\]
The density introduced in \eqref{eq:example-invariant-density} below is bounded above and below, so the same
integrability holds under \(\mu\). Once invariance of \(\mu\) is established
below, Proposition~\ref{prop:diagonal-strong-framework} therefore shows that
the corresponding stationary realizations are \(\Hh\)-valued continuous
strong solutions.

For this diagonal noise,
\[
    \Hh_{\Q}
    =
    \left\{
        v=\sum_{k=1}^{\infty}v_ke_k:
        \sum_{k=1}^{\infty}\frac{v_k^2}{q_k}<\infty
    \right\},
    \qquad
    \|v\|_{\Q^{-1}}^2
    =
    \sum_{k=1}^{\infty}\frac{v_k^2}{q_k}.
\]
In particular,
\[
    \|F(x)\|_{\Q^{-1}}^2
    =
    \sum_{k=1}^{\infty}q_ka_k^2\sin^2(x_k)
    =
    \|G(x)\|_{\Uh}^2.
\]

The Gaussian divergence has the representation
\begin{equation}\label{eq:example-gaussian-divergence}
\begin{aligned}
    \divnu F(x)=
    \Tr\bigl(\Q \nabla\theta(x)\bigr)
    +2\inner{A^*\theta(x)}{x}=
    \sum_{k=1}^{\infty}q_ka_k\cos(x_k)
    -2\sum_{k=1}^{\infty}\lambda_ka_kx_k\sin(x_k).
\end{aligned}
\end{equation}
The second series is absolutely convergent for every \(x\in\Hh\), and the
resulting divergence belongs to \(L^p(\nu)\) for every finite \(p\geq1\).

Define
\begin{equation}\label{eq:example-invariant-density}
    \rho(x):=Z_\Phi^{-1}e^{-\Phi(x)},
    \qquad \text{with}~
    Z_\Phi:=\int_{\Hh}e^{-\Phi(x)}\,\nu(\dd x).
\end{equation}
Since \(\Phi\) is bounded, \(0<Z_\Phi<\infty\), and \(\rho\) is bounded
above and below by positive constants. Furthermore,
\[
    \nabla\log\rho=-\nabla\Phi=2\theta,
    \qquad
    \nabla\rho=2\rho\theta,
    \qquad
    \Q \nabla\log\rho=2F.
\]
Since \(\log\rho=-\Phi-\log Z_\Phi\), we also have
\[
    \nabla^2(\log\rho)=2\nabla\theta,
    \qquad
    A^*\nabla(\log\rho)=2A^*\theta,
    \qquad
    \Q \nabla^2(\log\rho)=2\Q \nabla\theta.
\]
The estimates above show that \(\log\rho\in C_b^2(\Hh)\),
\((\nabla\log\rho)(\Hh)\subset\Dom(A^*)\), and
\(A^*\nabla(\log\rho)\in C_b(\Hh;\Hh)\); hence the smooth logarithmic-density
requirements in Assumption~\ref{ass:path-regularity} hold. Moreover,
\(\log\rho,\rho\in W^{1,2}(\Hh,\nu)\), and their weak gradients agree with
the displayed classical gradients. We next verify the closed-generator
identities explicitly.
Define the cylindrical functions
\[
    \psi_n(x):=-2\sum_{k=1}^na_k\cos(x_k)-\log Z_\Phi.
\]
Then \(\psi_n\to\log\rho\) uniformly, and
\[
    \Lzero\psi_n(x)
    =\sum_{k=1}^nq_ka_k\cos(x_k)
     -2\sum_{k=1}^n\lambda_ka_kx_k\sin(x_k).
\]
The first series converges uniformly. The second converges in
\(L^2(\nu)\), because
\(\sum_{k=1}^{\infty}\lambda_k^2a_k^2<\infty\) and
\(x\in L^2(\nu;\Hh)\). Closedness of \(\Lzero\) therefore gives
\begin{equation}\label{eq:example-L0-logrho}
    \log\rho\in\Dom(\Lzero),
    \qquad
    \Lzero(\log\rho)=\divnu F
    \quad\text{in }L^2(\nu).
\end{equation}

Similarly, set
\[
    \widehat\rho_n(x)
    :=Z_\Phi^{-1}\exp\left\{-2\sum_{k=1}^na_k\cos(x_k)\right\}.
\]
The functions \(\widehat\rho_n\) converge uniformly to \(\rho\), and direct
cylindrical differentiation gives
\(\Lzero\widehat\rho_n=\widehat\rho_n h_n\), where \(h_n\) is obtained by
truncating each series in the bracket below at \(n\):
\[
\begin{aligned}
    h(x)
    :=\sum_{k=1}^{\infty}q_ka_k\cos(x_k)
        -2\sum_{k=1}^{\infty}\lambda_ka_kx_k\sin(x_k) +
        2\sum_{k=1}^{\infty}q_ka_k^2\sin^2(x_k).
\end{aligned}
\]
The first and third series converge uniformly, while the second converges
in \(L^2(\nu)\); hence \(h_n\to h\) in \(L^2(\nu)\), where \(h\) is the
displayed bracket. Since
\(\widehat\rho_n\to\rho\) in \(L^\infty(\nu)\), it follows that
\(\widehat\rho_n h_n\to\rho h\) in \(L^2(\nu)\).
Thus \(\rho\in\Dom(\Lzero)\) and
\(\Lzero\rho=\rho h\) in \(L^2(\nu)\). Since
\(\rho\in W^{1,2}(\Hh,\nu)\cap L^\infty(\nu)\),
\(F\in\Dom(\divnu)\), and \(\inner{F}{\nabla\rho}\in L^2(\nu)\),
Lemma~\ref{lem:gaussian-divergence-product} gives
\[
    \divnu(\rho F)
    =
    \rho\divnu F+\inner{\nabla\rho}{F}
    =
    \Lzero\rho.
\]
For \(x\in\Dom(A)\), the classical Kolmogorov expression for
\(\log\rho\) is
\[
    \inner{Ax}{\nabla(\log\rho)(x)}
    +\frac12\Tr\bigl(\Q \nabla^2(\log\rho)(x)\bigr)
    =\divnu F(x),
\]
which agrees with the closed-generator value in~
\eqref{eq:example-L0-logrho}. We choose this Borel expression as the
representative of \(\Lzero(\log\rho)\). Thus the general derivation of~
\eqref{eq:logrho-ito} applies. Finally,
\[
    B^*\nabla(\log\rho)=2G,
    \qquad
    \inner{F}{\nabla(\log\rho)}=2\|G\|_{\Uh}^2,
\]
and, together with~\eqref{eq:example-L0-logrho}, all integrability
requirements used below hold in every finite \(L^p(\mu)\).

This infinitesimal identity can be upgraded to invariance without a
generator-core argument because the example is coordinatewise decoupled.
Indeed, the first \(n\) coordinates form an autonomous
\(n\)-dimensional diffusion whose invariant density relative to its
Gaussian Ornstein--Uhlenbeck law is
\[
    \rho_n(x_1,\ldots,x_n)
    =
    Z_n^{-1}
    \exp\left\{
        -2\sum_{k=1}^na_k\cos(x_k)
    \right\}.
\]
Here \(Z_n\) is the normalizing constant with respect to the Gaussian
marginal \(\nu_n\).
The usual finite-dimensional integration-by-parts identity gives detailed
balance for this projected gradient
diffusion~\cite{FukushimaOshimaTakeda, QW99}. By the product structure of \(\nu\) and
the absolute convergence of \(\sum_ka_k\), the first \(n\)-coordinate
marginal of \(\mu\) is exactly \(\rho_n\nu_n\). Therefore, for every bounded
function \(f\) depending on finitely many coordinate maps
\(x_1,\ldots,x_n\), the autonomous finite-dimensional dynamics gives
\[
    \int_{\Hh}P_tf\,\dd\mu=\int_{\Hh}f\,\dd\mu.
\]
These coordinate-cylinder functions generate \(\mathcal B(\Hh)\) and
determine probability measures on the separable space \(\Hh\). Hence the
well-posed Markov semigroup of~
\eqref{eq:example-nonlinear-equation} leaves \(\mu\) invariant.
Finally,
\[
    V=2F-\Q \nabla\log\rho=0,
    \qquad
    J_\mu=\rho F-\frac12\Q \nabla\rho=0.
\]
Therefore, Theorem~\ref{thm:reversibility-criteria} gives \(\EP=0\),
self-adjointness of the generator in \(L^2(\mu)\), and pathwise reversibility
of the stationary strong solution. For
\(\Hh=L^2(0,\pi)\) and \(A=\Delta_D\), one has \(\lambda_k=k^2\), so all
the summability conditions above hold.

\section{Proofs of the main results}\label{Sec3}

\subsection{Infinite-dimensional Girsanov transform}

We first compare the nonlinear dynamics~\eqref{eq:intro-spde} with the Ornstein--Uhlenbeck
reference process~\eqref{eq:ou-reference}. The sign of the exponential
density below is chosen so that the change of measure removes \(F\) from the
drift.

This is the standard Girsanov change of measure for stochastic evolution
equations driven by cylindrical Wiener noise;
see~\cite[Theorem~10.14 and Theorem~10.18]{DZ92}. A related change of
reference measure is used in the Onsager--Machlup analysis of stochastic
evolution equations in~\cite{BardinaRoviraTindel2003}. The purpose here is
different: we compare the stationary forward and reversed path laws, and the
relevant control is the minimal noise-space representative \(B^\dagger F\).

\begin{lemma}[Infinite-dimensional Girsanov transform]
\label{lem:girsanov}
Let \(0\leq s<t<\infty\), and let \(X\) be the adapted strong solution of the nonlinear dynamics~\eqref{eq:intro-spde},
\[
    \dd X_r=(AX_r+F(X_r))\,\dd r+B\,\dd W_r,
    \qquad r\in[s,t],
\]
under \(\Pp\), where \(F=BG\) for a bounded globally Lipschitz map
\(G:\Hh\to(\ker B)^\perp\subset\Uh\). For
\(r\in[s,t]\), define
\begin{equation}
\label{eq:girsanov-density}
    Z_{s,r}
    :=
    \exp\left\{
        -\int_s^r
         \langle G(X_u),\dd W_u\rangle_{\Uh}
        -\frac12\int_s^r\|G(X_u)\|_{\Uh}^2\,\dd u
    \right\}.
\end{equation}
Then \((Z_{s,r})_{r\in[s,t]}\) is a uniformly integrable martingale and
\[
    \E[Z_{s,t}\mid\mathcal F_s]=1.
\]
Define \(\widetilde{\Pp}\) on \(\mathcal F_t\) by
\[
    \frac{\dd\widetilde{\Pp}}{\dd\Pp}\bigg|_{\mathcal F_t}
    =Z_{s,t}.
\]
Then \(\widetilde{\Pp}\) and \(\Pp\) are equivalent on \(\mathcal F_t\),
and
\[
    \widetilde W_r-\widetilde W_s
    :=
    W_r-W_s+\int_s^rG(X_u)\,\dd u
\]
is a cylindrical Wiener process on \(\Uh\) under
\(\widetilde{\Pp}\). Moreover, \(X\) satisfies
\[
    \dd X_r=AX_r\,\dd r+B\,\dd\widetilde W_r
\]
on \([s,t]\), in the strong sense, and the change of measure does not alter
the law of \(X_s\).
\end{lemma}

\begin{proof}
Since \(X\) is adapted with continuous paths and \(G\) is continuous,
\(G(X_r)\) is progressively measurable. Its boundedness gives
\[
    \E\exp\left\{
        \frac12\int_s^t\|G(X_u)\|_{\Uh}^2\,\dd u
    \right\}
    \leq
    \exp\left\{
        \frac{t-s}{2}\|G\|_\infty^2
    \right\}<\infty.
\]
Novikov's condition therefore implies that~
\eqref{eq:girsanov-density} is a uniformly integrable stochastic
exponential. The Hilbert-space Girsanov theorem
\cite[Theorem~10.14 and Theorem~10.18]{DZ92} then gives the asserted
cylindrical Wiener process \(\widetilde W\).

Since \(F=BG\), substitution of
\(\dd W_r=\dd\widetilde W_r-G(X_r)\,\dd r\) yields
\[
\begin{aligned}
    \dd X_r
    &=
    \bigl(AX_r+F(X_r)\bigr)\,\dd r
    +B\bigl(\dd\widetilde W_r-G(X_r)\,\dd r\bigr) \\
    &=AX_r\,\dd r+B\,\dd\widetilde W_r.
\end{aligned}
\]
Finally, for every bounded \(\mathcal F_s\)-measurable random variable \(Y\),
\[
    \widetilde{\E}[Y]
    =
    \E\!\left[Y\,\E[Z_{s,t}\mid\mathcal F_s]\right]
    =
    \E[Y].
\]
Taking \(Y=f(X_s)\) proves the last assertion.
\end{proof}

\begin{remark}[Noise-range condition]
The factorization \(F=BG\) in Lemma~\ref{lem:girsanov} is the standard
noise-range condition for a Girsanov change of measure between stochastic
evolution equations~\cite{DZ92}; no bounded inverse of \(B\) is required.
Minimality of \(G\) and automatic validity of Novikov's condition under the
main assumptions are recorded in
Remark~\ref{rem:girsanov-admissibility}. Localized or
exponential-integrability conditions could replace boundedness, but are not
needed here.
\end{remark}

\subsection{Time reversal of the reversible Ornstein--Uhlenbeck reference}

Lemma~\ref{lem:girsanov} reduces the nonlinear dynamics to the reversible Ornstein--Uhlenbeck
reference process. The Ornstein--Uhlenbeck process is reversible when it starts from its
invariant measure \(\nu\); if it starts instead from a weighted measure
\(\eta\nu\), time reversal introduces the endpoint correction appearing
below.

Symmetry and reversibility of Hilbert-space Ornstein--Uhlenbeck semigroups
are characterized in~\cite{MG02}; see also the
survey~\cite{Bogachev2018}. Time reversal of more general infinite-dimensional
diffusions was studied under local finite-entropy assumptions
in~\cite{FollmerWakolbinger1986}; Millet, Nualart, and Sanz instead used
techniques from the stochastic calculus of
variations~\cite{MilletNualartSanz1989}. The role of time reversal here is
more specific: we reverse the symmetric Ornstein--Uhlenbeck reference law by detailed balance
and then recover the nonlinear law through Girsanov's theorem.

\begin{proposition}[Time reversal of the reversible Ornstein--Uhlenbeck reference]
\label{prop:OU-time-reversal}
Let \(Z\) be the continuous mild solution associated with
\[
    \dd Z_r=AZ_r\,\dd r+B\,\dd W_r,
    \qquad r\in[s,t],
\]
and suppose that Assumptions~\ref{ass:AFQ} and~\ref{ass:OU} hold. Let
\(\eta:\Hh\to(0,\infty)\) be measurable with
\[
    \int_{\Hh}\eta\,\dd\nu=1,
\]
and start the reference process from
\[
    Z_s\sim\eta\nu.
\]
Denote expectation for this Ornstein--Uhlenbeck process by \(\widetilde{\E}_\eta\). Then, for
every bounded measurable path functional
\(\Phi:C([s,t];\Hh)\to\R\),
\begin{equation}
\label{eq:OU-time-reversal}
    \widetilde{\E}_\eta
    \left[
        \Phi(Z_\cdot)\frac{\eta(Z_t)}{\eta(Z_s)}
    \right]
    =
    \widetilde{\E}_\eta
    \left[
        \Phi(\Theta_{s,t}Z_\cdot)
    \right].
\end{equation}
\end{proposition}

\begin{proof}
Let \(p_r(x,\dd y)\) be the transition kernel of the Ornstein--Uhlenbeck semigroup. Its
symmetry in \(L^2(\nu)\) is equivalent to the detailed-balance identity
\begin{equation}\label{eq:OU-detailed-balance-kernel}
    \nu(\dd x)\,p_r(x,\dd y)
    =
    \nu(\dd y)\,p_r(y,\dd x),
    \qquad r\geq0,
\end{equation}
as an identity of measures on \(\Hh\times\Hh\); see~\cite[Eq.~(1.2) and Theorem~2.4]{MG02}.

The endpoint weight is integrable even when \(\eta\) is unbounded. Indeed,
invariance of \(\nu\) gives
\[
\begin{aligned}
    \widetilde{\E}_\eta
    \left[\frac{\eta(Z_t)}{\eta(Z_s)}\right]
    &=
    \int_{\Hh\times\Hh}
    \frac{\eta(y)}{\eta(x)}\eta(x)\nu(\dd x)p_{t-s}(x,\dd y)
     \\
    &=\int_{\Hh}\eta(y)\nu(\dd y)=1.
\end{aligned}
\]
Thus the left-hand side of~\eqref{eq:OU-time-reversal} is absolutely
integrable for every bounded \(\Phi\).

It is enough to verify~\eqref{eq:OU-time-reversal} for bounded cylindrical
path functionals. Fix
\[
    s=t_0<t_1<\cdots<t_n=t
\]
and let
\[
    \Phi(Z_\cdot)=\prod_{k=0}^n f_k(Z_{t_k}),
\]
where the \(f_k\) are bounded and measurable. Using the initial law
\(\eta\nu\), followed by repeated applications of~
\eqref{eq:OU-detailed-balance-kernel}, we obtain
\[
\begin{aligned}
&\widetilde{\E}_\eta
 \left[
    \prod_{k=0}^n f_k(Z_{t_k})
    \frac{\eta(Z_t)}{\eta(Z_s)}
 \right] \\
&\quad=
\int_{\Hh^{n+1}}
    \prod_{k=0}^n f_k(x_k)\,
    \frac{\eta(x_n)}{\eta(x_0)}
    \eta(x_0)\nu(\dd x_0)
    p_{t_1-t_0}(x_0,\dd x_1)\cdots
    p_{t_n-t_{n-1}}(x_{n-1},\dd x_n) \\
&\quad=
\int_{\Hh^{n+1}}
    \prod_{k=0}^n f_k(x_k)\,
    \eta(x_n)\nu(\dd x_n)
    p_{t_n-t_{n-1}}(x_n,\dd x_{n-1})\cdots
    p_{t_1-t_0}(x_1,\dd x_0).
\end{aligned}
\]

To identify the last integral, set
\[
    y_j:=x_{n-j},
    \qquad
    u_j:=s+t-t_{n-j},
    \qquad 0\leq j\leq n.
\]
Then \(s=u_0<u_1<\cdots<u_n=t\) and
\[
    u_j-u_{j-1}=t_{n-j+1}-t_{n-j}.
\]
After relabeling the variables, the preceding integral becomes
\[
\begin{aligned}
&\int_{\Hh^{n+1}}
    \eta(y_0)\nu(\dd y_0)
    \prod_{j=1}^n p_{u_j-u_{j-1}}(y_{j-1},\dd y_j)
    \prod_{k=0}^n f_k(y_{n-k}) \\
&\qquad=
\widetilde{\E}_\eta
\left[
    \prod_{k=0}^n f_k(Z_{s+t-t_k})
\right]
=
\widetilde{\E}_\eta
\left[
    \Phi(\Theta_{s,t}Z_\cdot)
\right].
\end{aligned}
\]
This proves~\eqref{eq:OU-time-reversal} for cylindrical functionals. Since
\(\Hh\) is separable, such functionals generate the Borel
\(\sigma\)-algebra of \(C([s,t];\Hh)\); the result for every bounded
measurable \(\Phi\) follows from the monotone class theorem.
\end{proof}

In the path-law comparison below, Proposition~\ref{prop:OU-time-reversal} is
applied with \(\eta=\rho\), because the nonlinear process is started from
its stationary law \(\mu=\rho\nu\).

\subsection{Gaussian divergence identities}
\label{subsec:gaussian-divergence}

The Gaussian integration-by-parts argument establishing
\(F\in\Dom(\divnu)\), together with the explicit formula for \(\divnu F\),
is deferred to Appendix~\ref{app:gaussian-divergence-calculation}. We only
record here the product rule needed below; its closed-divergence formulation
makes the required integrability explicit.

\begin{lemma}[Product rule for Gaussian divergence]
\label{lem:gaussian-divergence-product}
Let \(Y\in\Dom(\divnu)\), and let
\(\psi\in W^{1,2}(\Hh,\nu)\cap L^\infty(\nu)\). Assume that
\[
    \inner{Y}{\nabla\psi}\in L^2(\nu).
\]
Then \(\psi Y\in\Dom(\divnu)\) and
\begin{equation}\label{eq:gaussian-divergence-product}
    \divnu(\psi Y)
    =
    \psi\,\divnu Y+\inner{Y}{\nabla\psi}
\end{equation}
in \(L^2(\nu)\).
\end{lemma}

\begin{proof}
Set
\[
    g:=\psi\,\divnu Y+\inner{Y}{\nabla\psi}.
\]
The assumptions imply that \(g\in L^2(\nu)\) and
\(\psi Y\in L^2(\nu;\Hh)\). Let
\(\varphi\in W^{1,2}(\Hh,\nu)\cap L^\infty(\nu)\). Then
\(\psi\varphi\in W^{1,2}(\Hh,\nu)\), with
\[
    \nabla(\psi\varphi)=\psi \nabla\varphi+\varphi \nabla\psi.
\]
Using the defining adjoint relation for \(\divnu Y\), we obtain
\[
\begin{aligned}
    \int_{\Hh}\inner{\psi Y}{\nabla\varphi}\,\dd\nu
    &=
    \int_{\Hh}\inner{Y}{\nabla(\psi\varphi)}\,\dd\nu
    -\int_{\Hh}\varphi\inner{Y}{\nabla\psi}\,\dd\nu \\
    &=
    -\int_{\Hh}\psi\varphi\,\divnu Y\,\dd\nu
    -\int_{\Hh}\varphi\inner{Y}{\nabla\psi}\,\dd\nu \\
    &=
    -\int_{\Hh}\varphi g\,\dd\nu.
\end{aligned}
\]
Bounded functions in \(W^{1,2}(\Hh,\nu)\) are dense in
\(W^{1,2}(\Hh,\nu)\), for example by scalar truncation
\cite[Chapter~5]{Bogachev1998}. Since both sides of
the last identity are continuous in the \(W^{1,2}\)-norm, the identity
extends to every \(\varphi\in W^{1,2}(\Hh,\nu)\). By
Definition~\ref{def:gaussian-divergence}, this proves that
\(\psi Y\in\Dom(\divnu)\) and \(\divnu(\psi Y)=g\).
\end{proof}

Under Assumption~\ref{ass:path-regularity}, the fixed representative
\(\rho\) and its gradient \(\nabla\rho=\rho \nabla\log\rho\) are bounded. Hence
\[
    \rho\,\divnu F\in L^2(\nu),
    \qquad
    \inner{F}{\nabla\rho}\in L^2(\nu).
\]
Applying Lemma~\ref{lem:gaussian-divergence-product} with \(Y=F\) and
\(\psi=\rho\) therefore gives the closed product identity
\begin{equation}\label{eq:rhoF-closed-product}
    \divnu(\rho F)
    =\rho\,\divnu F+\inner{F}{\nabla\rho}
    \quad\text{in }L^2(\nu).
\end{equation}
This explicit product identity refines the divergence-domain conclusion in
Lemma~\ref{lem:density-regularity}. Together with the stationary
identity~\eqref{eq:FP}, it will be used below to obtain the logarithmic
stationary Fokker--Planck equation.

\subsection{Time reversal of the Girsanov density}

By Assumption~\ref{ass:path-regularity}, on every finite interval
\([s,t]\) the stationary process admits the \(\Hh\)-valued semimartingale
decomposition, for \(s\leq r\leq t\),
\[
    X_r
    =
    X_s+\int_s^r\bigl(AX_u+F(X_u)\bigr)\,\dd u+M_r^X,
    \quad \text{with}
     ~M_r^X:=\int_s^rB\,\dd W_u.
\]
Reversing the Girsanov density requires replacing the left-point
discretization of the stochastic line integral
\(\int_s^t\inner{\theta(X_r)}{\dd X_r}\) by its right-point discretization.
For a partition \(\pi=\{s=t_0<\cdots<t_N=t\}\), the two sums are
\[
    \sum_{j=0}^{N-1}\inner{\theta(X_{t_j})}{X_{t_{j+1}}-X_{t_j}}
    \quad\text{and}\quad
    \sum_{j=0}^{N-1}\inner{\theta(X_{t_{j+1}})}{X_{t_{j+1}}-X_{t_j}}.
\]
Their difference is an increment-product sum for the pair
\((\theta(X),X)\). The next lemma identifies the limit of such sums for
general Hilbert-space semimartingales; it will subsequently be applied with
\(Y=\theta(X)\) and \(Z=X\).

This left--right correction serves here to reverse the stochastic line
integral and, together with the linear Ornstein--Uhlenbeck term, yields the Gaussian
divergence in~\eqref{divF}. This differs from the Onsager--Machlup small-tube
problem considered in~\cite{BardinaRoviraTindel2003}, despite the appearance
of a trace correction in both settings.

\begin{lemma}[Left--right Riemann-sum correction]
\label{lem:left-right-riemann-correction}
Let \(0\leq a<b<\infty\), and let \(Y\) and \(Z\) be continuous
\(\Hh\)-valued semimartingales on \([a,b]\) whose continuous
local-martingale parts satisfy, for \(a\leq r\leq b\),
\[
    M_r^Y-M_a^Y=\int_a^r K_u\,\dd W_u,
    \qquad
    M_r^Z-M_a^Z=\int_a^r L_u\,\dd W_u,
\]
where \(K\) and \(L\) are predictable
\(\mathcal L_2(\Uh,\Hh)\)-valued processes satisfying
\[
    \int_a^b
    \bigl(\|K_u\|_{\mathcal L_2(\Uh,\Hh)}^2
          +\|L_u\|_{\mathcal L_2(\Uh,\Hh)}^2\bigr)\,\dd u
    <\infty
    \qquad\text{almost surely.}
\]
Then,
for every sequence of deterministic partitions
\[
    \pi_n=\{a=t_0^n<t_1^n<\cdots<t_{N_n}^n=b\},
    \qquad \text{with mesh size}~|\pi_n|:=\max_{0\leq j\leq N_n-1}(t_{j+1}^n-t_j^n)\longrightarrow0,
\]
one has
\[
\begin{aligned}
\sum_{j=0}^{N_n-1}
\left\langle
    Y_{t_{j+1}^n}-Y_{t_j^n},
    Z_{t_{j+1}^n}-Z_{t_j^n}
\right\rangle_{\Hh} \longrightarrow
\int_a^b
\operatorname{Tr}_{\Uh}(K_u^*L_u)\,\dd u
\end{aligned}
\]
in probability as the mesh size tends to zero.
\end{lemma}

\begin{proof}
Write
\[
    Y_r=Y_a+A_r^Y+M_r^Y-M_a^Y,
    \qquad
    Z_r=Z_a+A_r^Z+M_r^Z-M_a^Z,
\]
where \(A_a^Y=A_a^Z=0\) and \(A^Y,A^Z\) are continuous adapted
\(\Hh\)-valued finite-variation processes. Since the product of two
Hilbert--Schmidt operators is trace class,
\[
\begin{aligned}
    \int_a^b\bigl|\Tr_{\Uh}(K_u^*L_u)\bigr|\,\dd u
    &\leq
    \left(\int_a^b\|K_u\|_{\mathcal L_2(\Uh,\Hh)}^2\,\dd u\right)^{1/2}
    \left(\int_a^b\|L_u\|_{\mathcal L_2(\Uh,\Hh)}^2\,\dd u\right)^{1/2}
    <\infty
\end{aligned}
\]
almost surely. Thus the trace integral is well defined.

For \(R>0\), define
\[
\begin{aligned}
    \Gamma_r&:={}
      \sup_{a\leq v\leq r}
      \bigl(\|Y_v-Y_a\|_{\Hh}+\|Z_v-Z_a\|_{\Hh}\bigr)
      +\operatorname{Var}_{[a,r]}(A^Y)
      +\operatorname{Var}_{[a,r]}(A^Z)\\
      &\qquad+
      \int_a^r
      \bigl(\|K_u\|_{\mathcal L_2}^2
            +\|L_u\|_{\mathcal L_2}^2\bigr)\,\dd u,
\end{aligned}
\]
and
\[
    \tau_R:=\inf\{r\in[a,b]:\Gamma_r\geq R\}\wedge b.
\]
Then \(\Pp(\tau_R<b)\to0\), as \(R\to\infty\). Set
\[
    Y^R_r:=Y_{r\wedge\tau_R},\quad Z^R_r:=Z_{r\wedge\tau_R},\qquad
    K^R_u:=\mathbf 1_{\{u\leq\tau_R\}}K_u,\quad
    L^R_u:=\mathbf 1_{\{u\leq\tau_R\}}L_u,
\]
and stop the finite-variation parts in the same way. Until the final step,
we work with these stopped processes and suppress the superscript \(R\). The
Hilbert-space It\^o product formula for the stopped pair gives
\begin{equation}\label{eq:hilbert-product-formula}
\begin{aligned}
    \langle Y_b,Z_b\rangle_{\Hh}-\langle Y_a,Z_a\rangle_{\Hh}
    =
      \int_a^b\langle Z_u,\dd Y_u\rangle_{\Hh}
      +\int_a^b\langle Y_u,\dd Z_u\rangle_{\Hh}+
      \int_a^b\Tr_{\Uh}(K_u^*L_u)\,\dd u .
\end{aligned}
\end{equation}
Indeed, for any orthonormal basis \((f_m)_{m\geq1}\) of \(\Uh\), the
quadratic-covariation term is
\[
    \sum_{m=1}^{\infty}
      \langle K_uf_m,L_uf_m\rangle_{\Hh}
    =\Tr_{\Uh}(K_u^*L_u).
\]
The series is absolutely convergent by the Hilbert--Schmidt
Cauchy--Schwarz inequality; see also the standard Hilbert-space product
formula in~\cite[Chapter~4]{DZ92}.

For \(u\in(t_j^n,t_{j+1}^n]\), define the predictable left-step
approximations
\[
    Y_u^{(n)}:=Y_{t_j^n},
    \qquad
    Z_u^{(n)}:=Z_{t_j^n},
    \qquad
    Y_a^{(n)}:=Y_a,\quad Z_a^{(n)}:=Z_a.
\]
Continuity and \(|\pi_n|\to0\) imply that \(Y^{(n)}\to Y\) and
\(Z^{(n)}\to Z\) uniformly on \([a,b]\), almost surely. For the first
left-point sum, the finite-variation error obeys
\[
    \left|
      \int_a^b\langle Y_u^{(n)}-Y_u,\dd A_u^Z\rangle_{\Hh}
    \right|
    \leq
    \sup_{a\leq u\leq b}\|Y_u^{(n)}-Y_u\|_{\Hh}
    \operatorname{Var}_{[a,b]}(A^Z)
    \longrightarrow0
\]
almost surely, while the It\^o isometry gives
\[
\begin{aligned}
    \E\left|
      \int_a^b
      \langle L_u^*(Y_u^{(n)}-Y_u),\dd W_u\rangle_{\Uh}
    \right|^2
    &\leq
    \E\left[
      \sup_{a\leq u\leq b}\|Y_u^{(n)}-Y_u\|_{\Hh}^2
      \int_a^b\|L_u\|_{\mathcal L_2}^2\,\dd u
    \right]
    \longrightarrow0
\end{aligned}
\]
by dominated convergence. Indeed, after stopping,
\[
    \sup_{a\leq u\leq b}\|Y_u^{(n)}-Y_u\|_{\Hh}\leq2R,
    \qquad
    \int_a^b\|L_u\|_{\mathcal L_2}^2\,\dd u\leq R,
\]
so the random variable in the last expectation is bounded by \(4R^3\).
Hence
\[
    \sum_{j=0}^{N_n-1}
      \langle Y_{t_j^n},Z_{t_{j+1}^n}-Z_{t_j^n}\rangle_{\Hh}
    \longrightarrow
    \int_a^b\langle Y_u,\dd Z_u\rangle_{\Hh}
\]
in probability for the stopped processes. Interchanging \(Y\) and \(Z\)
gives
\[
    \sum_{j=0}^{N_n-1}
      \langle Z_{t_j^n},Y_{t_{j+1}^n}-Y_{t_j^n}\rangle_{\Hh}
    \longrightarrow
    \int_a^b\langle Z_u,\dd Y_u\rangle_{\Hh}.
\]

Finally, summing the elementary identity
\[
\begin{aligned}
    \langle Y_{t_{j+1}^n},Z_{t_{j+1}^n}\rangle_{\Hh}
      -\langle Y_{t_j^n},Z_{t_j^n}\rangle_{\Hh}
    &=
      \langle Y_{t_j^n},Z_{t_{j+1}^n}-Z_{t_j^n}\rangle_{\Hh} \\
    &\quad+
      \langle Z_{t_j^n},Y_{t_{j+1}^n}-Y_{t_j^n}\rangle_{\Hh} \\
    &\quad+
      \langle
        Y_{t_{j+1}^n}-Y_{t_j^n},
        Z_{t_{j+1}^n}-Z_{t_j^n}
      \rangle_{\Hh}
\end{aligned}
\]
over \(j\), and comparing the limit with
\eqref{eq:hilbert-product-formula}, proves the assertion for the stopped
processes. If \(S_n\) and \(I\) denote the increment sum and trace integral
in the statement, and \(S_n^R,I^R\) their stopped versions, then
\[
    \limsup_{n\to\infty}\Pp(|S_n-I|>\varepsilon)
    \leq \Pp(\tau_R<b),
    \qquad \varepsilon>0.
\]
Letting \(R\to\infty\) proves the claim.
\end{proof}

We now apply the left--right correction to \(\theta(X)\) and combine it
with the Gaussian divergence formula to reverse the Girsanov density.

\begin{proposition}[Time reversal of the Girsanov density]
\label{prop:time-reversal-girsanov}
Let Assumptions~\ref{ass:AFQ},~\ref{ass:OU},~\ref{ass:density},
and~\ref{ass:path-regularity} hold, and let \(X\) denote the stationary strong
realization of Assumption~\ref{ass:path-regularity}. Throughout the proposition,
\(\divnu F\) denotes the Borel representative fixed
in~\eqref{eq:divF-borel-representative}. Set
\[
    G:=B^*\theta=B^\dagger F.
\]
Let \(Z_{s,t}\) be the Girsanov density
from~\eqref{eq:girsanov-density}. For every fixed \(0\leq s<t<\infty\), there
exists, for this stationary realization, a Borel function
\[
    \mathcal Z_{s,t}:C([s,t];\Hh)\to(0,\infty)
\]
such that \(\mathcal Z_{s,t}(X)=Z_{s,t}\) almost surely and
\begin{equation}\label{eq:thetaZ}
\begin{aligned}
    \bigl(\mathcal Z_{s,t}\circ\Theta_{s,t}\bigr)(X)
    =
    \exp\Bigg\{
        \int_s^t
        \langle G(X_r),\dd W_r\rangle_{\Uh} +
        \int_s^t
        \left[
            \divnu F(X_r)
            +\frac32\|G(X_r)\|_{\Uh}^2
        \right]\dd r
    \Bigg\}
\end{aligned}
\end{equation}
almost surely. We henceforth denote the left-hand side of
\eqref{eq:thetaZ} by \(\Theta_{s,t}^*Z_{s,t}\). Moreover,
\begin{equation}\label{eq:F-energy-control}
    \|G(x)\|_{\Uh}
    =
    \|F(x)\|_{\Q^{-1}},
    \qquad x\in\Hh.
\end{equation}
\end{proposition}

\begin{proof}
Fix \(s<t\) and use the reflection-invariant dyadic partitions
\[
    t_j^n:=s+\frac{j}{2^n}(t-s),
    \qquad
    0\leq j\leq2^n.
\]
For \(\omega\in C([s,t];\Hh)\), define the Borel left- and right-point sums
\[
\begin{aligned}
    S_n^-(\theta;\omega)
    &:=
    \sum_{j=0}^{2^n-1}
    \inner{\theta(\omega_{t_j^n})}
          {\omega_{t_{j+1}^n}-\omega_{t_j^n}}, \\
    S_n^+(\theta;\omega)
    &:=
    \sum_{j=0}^{2^n-1}
    \inner{\theta(\omega_{t_{j+1}^n})}
          {\omega_{t_{j+1}^n}-\omega_{t_j^n}}.
\end{aligned}
\]
Since \(X\) is a continuous \(\Hh\)-valued semimartingale and
\(\theta\in C_b^2(\Hh;\Hh)\), the left-point sums converge in probability
to
\[
    I_{s,t}^-(\theta;X)
    :=
    \int_s^t\inner{\theta(X_r)}{\dd X_r}.
\]
This is the standard Riemann-sum construction of the Hilbert-space It\^o
integral \cite[Chapter~4]{DZ92}.
By the strong semimartingale decomposition and \(G=B^*\theta\),
\begin{equation}\label{eq:left-line-integral}
\begin{aligned}
    I_{s,t}^-(\theta;X)
    =
    \int_s^t
    \left[
        \inner{A^*\theta(X_r)}{X_r}
        +\inner{\theta(X_r)}{F(X_r)}
    \right]\dd r+
    \int_s^t
    \langle G(X_r),\dd W_r\rangle_{\Uh}.
\end{aligned}
\end{equation}
Since \(X_r\in\Dom(A)\) for
\(\dd r\otimes\Pp\)-almost every \((r,\omega)\) and
\(\theta(X_r)\in\Dom(A^*)\), operator duality gives
\[
    \inner{\theta(X_r)}{AX_r}
    =\inner{A^*\theta(X_r)}{X_r}
\]
for \(\dd r\otimes\Pp\)-almost every \((r,\omega)\). Both sides are
time integrable by the boundedness of \(\theta\) and \(A^*\theta\), the
local integrability of \(AX\), and the continuity of \(X\).

The difference between the right- and left-point sums is
\[
\begin{aligned}
    S_n^+(\theta;X)-S_n^-(\theta;X)
    =
    \sum_{j=0}^{2^n-1}
    \inner{
        \theta(X_{t_{j+1}^n})-\theta(X_{t_j^n})
    }{
        X_{t_{j+1}^n}-X_{t_j^n}
    }.
\end{aligned}
\]
Let \((f_j)_{j\geq1}\) and \((e_k)_{k\geq1}\) be orthonormal bases of
\(\Uh\) and \(\Hh\), respectively. The Hilbert-space It\^o formula gives
\[
\begin{aligned}
    \theta(X_r)
    =
    \theta(X_s)
    +\int_s^r\nabla\theta(X_u)\bigl(AX_u+F(X_u)\bigr)\,\dd u 
    +\frac12\int_s^r
       \sum_{j=1}^{\infty}
       \nabla^2\theta(X_u)[Bf_j,Bf_j]\,\dd u
    +M_r^\theta,
\end{aligned}
\]
where
\[
    M_r^\theta
    :=\int_s^r\nabla\theta(X_u)B\,\dd W_u.
\]
The second-order series is absolutely convergent in \(\Hh\), uniformly in
its state argument, because
\[
    \sum_{j=1}^{\infty}
    \bigl\|\nabla^2\theta(x)[Bf_j,Bf_j]\bigr\|_{\Hh}
    \leq
    \|\nabla^2\theta\|_\infty
    \|B\|_{\mathcal L_2(\Uh,\Hh)}^2.
\]
Moreover,
\[
    \int_s^t
    \|\nabla\theta(X_u)B\|_{\mathcal L_2(\Uh,\Hh)}^2\,\dd u
    \leq
    (t-s)\|\nabla\theta\|_\infty^2
    \|B\|_{\mathcal L_2(\Uh,\Hh)}^2,
\]
so \(M^\theta\) is a continuous square-integrable \(\Hh\)-valued
martingale.

For the present pair of continuous martingales, define the scalar trace
covariation by
\[
    [M^\theta,M^X]^{\mathrm{tr}}_r
    :=
    \lim_{m\to\infty}
    \sum_{k=1}^m
    \bigl[
       \inner{M^\theta}{e_k},
       \inner{M^X}{e_k}
    \bigr]_r,
\]
where the limit is uniform on compact time intervals in probability. The
limit is basis independent. A coordinatewise covariation calculation gives
\[
\begin{aligned}
    [M^\theta,M^X]^{\mathrm{tr}}_r
    &=
    \int_s^r\sum_{j=1}^{\infty}
    \inner{\nabla\theta(X_u)Bf_j}{Bf_j}\,\dd u \\
    &=
    \int_s^r
    \Tr_{\Uh}\!\left(B^*\nabla\theta(X_u)B\right)\dd u \\
    &=
    \int_s^r
    \Tr_{\Hh}\!\left(\Q \nabla\theta(X_u)\right)\dd u.
\end{aligned}
\]
Indeed,
\[
    \left|
        \Tr_{\Uh}\!\left(B^*\nabla\theta(x)B\right)
    \right|
    \leq
    \|\nabla\theta\|_\infty
    \|B\|_{\mathcal L_2(\Uh,\Hh)}^2.
\]
Applying Lemma~\ref{lem:left-right-riemann-correction} with
\[
    Y=\theta(X),
    \qquad Z=X,
    \qquad K=\nabla\theta(X)B,
    \qquad L=B,
\]
is legitimate by the estimates above. Since the Hilbert spaces are real,
\[
\begin{aligned}
    \Tr_{\Uh}\!\left((\nabla\theta(X_u)B)^*B\right)
    &=\sum_{j=1}^{\infty}
      \inner{\nabla\theta(X_u)Bf_j}{Bf_j} \\
    &=\Tr_{\Uh}\!\left(B^*\nabla\theta(X_u)B\right).
\end{aligned}
\]
Thus lemma \ref{lem:left-right-riemann-correction} shows that the difference of the right- and left-point sums
converges in probability to the trace term above. Consequently,
\begin{equation}\label{eq:left-right-line-integrals}
    I_{s,t}^+(\theta;X)
    :=
    \lim_{n\to\infty}S_n^+(\theta;X)
    =
    I_{s,t}^-(\theta;X)
    +
    \int_s^t
    \Tr\!\left(\Q \nabla\theta(X_r)\right)\dd r
\end{equation}
in probability. Choose a deterministic subsequence, not relabeled, along
which both \(S_n^-(\theta;X)\) and \(S_n^+(\theta;X)\) converge almost
surely. This subsequence is fixed for the stationary path law under
consideration. Define
\[
\begin{aligned}
    \mathcal I_{s,t}^-(\omega)
    &:=
    \lim_{n\to\infty}S_n^-(\theta;\omega), \\
    \mathcal I_{s,t}^+(\omega)
    &:=
    \lim_{n\to\infty}S_n^+(\theta;\omega)
\end{aligned}
\]
on the Borel set where the respective limits exist, and set them equal to
zero elsewhere. These are Borel path functionals and
\[
    \mathcal I_{s,t}^\pm(X)=I_{s,t}^\pm(\theta;X)
    \qquad\text{almost surely}.
\]
The dyadic partitions are invariant under \(r\mapsto s+t-r\), and direct
reindexing gives, for every path \(\omega\),
\[
    S_n^-(\theta;\Theta_{s,t}\omega)
    =
    -S_n^+(\theta;\omega).
\]
Taking limits along the chosen subsequence and using~
\eqref{eq:left-right-line-integrals} yields, almost surely,
\begin{equation}\label{eq:left-integral-reversal}
\begin{aligned}
    \mathcal I_{s,t}^-(\Theta_{s,t}X)
    =
    -\mathcal I_{s,t}^-(X)-
    \int_s^t
    \Tr\!\left(\Q \nabla\theta(X_r)\right)\dd r
\end{aligned}
\end{equation}

Since \(F=\Q\theta=BB^*\theta\), one has
\[
    \inner{\theta}{F}
    =
    \langle B^*\theta,B^*\theta\rangle_{\Uh}
    =
    \|G\|_{\Uh}^2.
\]
Thus~\eqref{eq:left-line-integral} implies
\[
\begin{aligned}
    \int_s^t\langle G(X_r),\dd W_r\rangle_{\Uh}
    =
    I_{s,t}^-(\theta;X)-\int_s^t\inner{A^*\theta(X_r)}{X_r}\,\dd r-
    \int_s^t\|G(X_r)\|_{\Uh}^2\,\dd r.
\end{aligned}
\]
Substituting this identity into~\eqref{eq:girsanov-density} gives, almost
surely,
\begin{equation}\label{eq:canonical-logZ}
\begin{aligned}
    \log Z_{s,t}
    =
    -\mathcal I_{s,t}^-(X)
    +\int_s^t\inner{A^*\theta(X_r)}{X_r}\,\dd r +
    \frac12\int_s^t\|G(X_r)\|_{\Uh}^2\,\dd r.
\end{aligned}
\end{equation}
Define the Borel state-path functional
\[
\begin{aligned}
    \log\mathcal Z_{s,t}(\omega)
    :=
    -\mathcal I_{s,t}^-(\omega)
    +\int_s^t
      \inner{A^*\theta(\omega_r)}{\omega_r}\,\dd r +
    \frac12\int_s^t\|G(\omega_r)\|_{\Uh}^2\,\dd r.
\end{aligned}
\]
Then \(\mathcal Z_{s,t}(X)=Z_{s,t}\) almost surely by~
\eqref{eq:canonical-logZ}. The ordinary time integrals are invariant under
\(\Theta_{s,t}\). Applying~\eqref{eq:left-integral-reversal} and then
using~\eqref{eq:left-line-integral} once more yields
\[
\begin{aligned}
    \log\bigl(\mathcal Z_{s,t}(\Theta_{s,t}X)\bigr)
    &=
    \int_s^t\langle G(X_r),\dd W_r\rangle_{\Uh} \\
    &\quad+
    \int_s^t
    \left[
        \Tr\!\left(\Q \nabla\theta(X_r)\right)
        +2\inner{A^*\theta(X_r)}{X_r}
        +\frac32\|G(X_r)\|_{\Uh}^2
    \right]\dd r.
\end{aligned}
\]
The first two finite-variation terms equal
\(\divnu F(X_r)\) by~\eqref{divF}; exponentiation proves~
\eqref{eq:thetaZ}.

Finally, \(G=B^*\theta\in(\ker B)^\perp\) and \(BG=F\), so \(G\) is the
minimal-norm control representing \(F\). Definition~\ref{def:CM-space}
therefore gives~\eqref{eq:F-energy-control}.
\end{proof}

\subsection{Path-law ratio}

Combining the reversed Girsanov density with reversibility of the Ornstein--Uhlenbeck
reference gives the finite-time likelihood ratio needed below.

\begin{proposition}[Forward-backward path-law ratio]
\label{prop:path-ratio}
Under Assumptions~\ref{ass:AFQ},~\ref{ass:OU},
\ref{ass:density}, and~\ref{ass:path-regularity}, for
\(0\le s<t<\infty\), let \(X\) be the stationary strong realization of
Assumption~\ref{ass:path-regularity}, so that \(X_s\sim\mu\), and let
\(\Omega_{s,t}:=C([s,t];\Hh)\). Then
\begin{equation}
\label{eq:path-ratio}
    \frac{\dd\Pp^-_{[s,t]}}
         {\dd\Pp^+_{[s,t]}}(\omega)
    =
    \frac{\mathcal Z_{s,t}(\omega)}
         {\mathcal Z_{s,t}(\Theta_{s,t}\omega)}
    \frac{\rho(\omega_t)}{\rho(\omega_s)}
\end{equation}
for \(\Pp^+_{[s,t]}\)-almost every \(\omega\in\Omega_{s,t}\), where
\(\mathcal Z_{s,t}\) is the Borel path functional constructed in
Proposition~\ref{prop:time-reversal-girsanov}.
\end{proposition}

\begin{proof}
Let \(R_{[s,t]}\) denote the law on \(\Omega_{s,t}\) of \(X\) under the
Girsanov-transformed measure \(\widetilde{\Pp}\). Lemma~\ref{lem:girsanov}
and Proposition~\ref{prop:time-reversal-girsanov} give
\begin{equation}\label{eq:OU-path-law-density}
    \frac{\dd R_{[s,t]}}{\dd\Pp^+_{[s,t]}}(\omega)
    =\mathcal Z_{s,t}(\omega).
\end{equation}
Under \(R_{[s,t]}\), the coordinate process is the reversible Ornstein--Uhlenbeck reference
process started from \(\mu=\rho\nu\).

Identity~\eqref{eq:OU-time-reversal} extends from bounded functionals to
nonnegative measurable functionals by monotone convergence. We may
therefore first take \(\Phi\geq0\) and define
\[
    \Psi(\omega)
    :=
    \frac{\Phi(\omega)}{\mathcal Z_{s,t}(\Theta_{s,t}\omega)}.
\]
Since \(\Theta_{s,t}^2\) is the identity,
Proposition~\ref{prop:OU-time-reversal}, applied under \(R_{[s,t]}\) to
\(\Psi\), gives
\[
\begin{aligned}
    \int_{\Omega_{s,t}}\Phi\,\dd\Pp^-_{[s,t]}
    &=\int_{\Omega_{s,t}}
      \frac{\Phi(\Theta_{s,t}\omega)}
           {\mathcal Z_{s,t}(\omega)}\,R_{[s,t]}(\dd\omega) \\
    &=\int_{\Omega_{s,t}}
      \frac{\Phi(\omega)}
           {\mathcal Z_{s,t}(\Theta_{s,t}\omega)}
      \frac{\rho(\omega_t)}{\rho(\omega_s)}
      \,R_{[s,t]}(\dd\omega) \\
    &=\int_{\Omega_{s,t}}\Phi(\omega)
      \frac{\mathcal Z_{s,t}(\omega)}
           {\mathcal Z_{s,t}(\Theta_{s,t}\omega)}
      \frac{\rho(\omega_t)}{\rho(\omega_s)}
      \,\Pp^+_{[s,t]}(\dd\omega),
\end{aligned}
\]
where the first equality uses the definition of \(\Pp^-_{[s,t]}\) and~
\eqref{eq:OU-path-law-density}, and the last uses~
\eqref{eq:OU-path-law-density} once more.
Applying the argument to the positive and negative parts gives the same
identity for every bounded signed \(\Phi\), and hence~
\eqref{eq:path-ratio}. Its right-hand side is strictly positive and finite
\(\Pp^+_{[s,t]}\)-almost surely, so the two path laws are equivalent.
\end{proof}

Formula~\eqref{eq:path-ratio} has the same endpoint-plus-action structure as
the path-space identities used in finite-dimensional entropy-production
theory~\cite{QW99,LebowitzSpohn1999}. Here the endpoint density is taken
relative to the Gaussian reference measure \(\nu\), and the stochastic
action is expressed in the noise Cameron--Martin directions; no
infinite-dimensional analogue of Lebesgue measure is invoked.

\subsection{Proof of the entropy production formula}
\label{subsec:entropy-production-proof}

\begin{proof}[Proof of Theorem~\ref{thm:main}]
We first derive the logarithmic stationary equation. By
Lemma~\ref{lem:density-regularity}, \(\rho\) and \(\log\rho\) satisfy the closed
generator chain rule~\eqref{eq:L0-rho-chain-rule} and
\(\nabla\rho=\rho \nabla\log\rho\).
By~\eqref{eq:rhoF-closed-product} and this gradient identity,
\[
    \divnu(\rho F)
    =\rho\left(\divnu F+\inner{F}{\nabla\log\rho}\right)
    \quad\text{in }L^2(\nu).
\]
Combining this identity with~\eqref{eq:L0-rho-chain-rule} and the closed
stationary equation~\eqref{eq:FP} gives
\begin{equation}
\label{eq:logFP-final}
    \Lzero(\log\rho)
    +\frac12\|B^*\nabla\log\rho\|_{\Uh}^2
    -\divnu F
    -\inner{F}{\nabla\log\rho}=0
    \quad\text{in }L^2(\nu).
\end{equation}
Indeed, division by \(\rho\) is legitimate because
\(e^{-\|\log\rho\|_\infty}\leq\rho\leq
e^{\|\log\rho\|_\infty}\).
The same bounds show that~\eqref{eq:logFP-final} also holds in
\(L^2(\mu)\). After fixing Borel representatives, stationarity and
Fubini's theorem imply that it holds at \(X_r\) for
\(\dd r\otimes\Pp\)-almost every \((r,\omega)\); hence it may be used
inside the time integrals below.

Set
\[
    G:=B^*\theta,
    \qquad
    \Gamma:=2G-B^*\nabla\log\rho=B^*(2\theta-\nabla\log\rho),
    \qquad
    \widehat V:=2F-\Q \nabla\log\rho.
\]
Then \(B\Gamma=\widehat V\) and
\(\Gamma(x)\in(\ker B)^\perp\). Hence \(\Gamma(x)\) is the unique
minimal-norm control representing \(\widehat V(x)\), and
\begin{equation}\label{eq:irreversible-control-energy}
    \|\Gamma(x)\|_{\Uh}^2
    =\|\widehat V(x)\|_{\Q^{-1}}^2,
    \qquad x\in\Hh.
\end{equation}
By~\eqref{eq:log-gradient-consistency}, \(\widehat V=V\)
\(\mu\)-almost everywhere. Moreover,
\[
    g(x):=\|\Gamma(x)\|_{\Uh}^2
         =\|\widehat V(x)\|_{\Q^{-1}}^2
\]
defines a function \(g\in C_b(\Hh)\).

Fix \(t\geq0\) and \(h>0\). Proposition~\ref{prop:path-ratio}, together
with~\eqref{eq:girsanov-density} and~\eqref{eq:thetaZ}, gives
\[
\begin{aligned}
    \log\left(\frac{\dd\Pp^-_{[t,t+h]}}
    {\dd\Pp^+_{[t,t+h]}}\right)(X)
    &=-2\int_t^{t+h}\langle G(X_r),\dd W_r\rangle_{\Uh} \\
    &\quad-\int_t^{t+h}
      \bigl[2\|G\|_{\Uh}^2+\divnu F\bigr](X_r)\,\dd r
      +\log\rho(X_{t+h})-\log\rho(X_t).
\end{aligned}
\]
Using the It\^o identity~\eqref{eq:logrho-ito} and then
\eqref{eq:logFP-final}, the finite-variation coefficient becomes
\[
\begin{aligned}
    -2\|G\|_{\Uh}^2-\divnu F
      +\Lzero(\log\rho)+\inner{F}{\nabla\log\rho} &=-2\|G\|_{\Uh}^2+2\inner{F}{\nabla\log\rho}
      -\frac12\|B^*\nabla\log\rho\|_{\Uh}^2\\
      &=-\frac12\|\Gamma\|_{\Uh}^2.
\end{aligned}
\]
Here
\(\langle G,B^*\nabla\log\rho\rangle_{\Uh}
  =\langle F,\nabla\log\rho\rangle_{\Hh}\).
After inverting the likelihood ratio, we obtain
\[
    \log\left(\frac{\dd\Pp^+_{[t,t+h]}}{\dd\Pp^-_{[t,t+h]}}\right)(X)
    =\int_t^{t+h}\langle\Gamma(X_r),\dd W_r\rangle_{\Uh}
     +\frac12\int_t^{t+h}g(X_r)\,\dd r.
\]
Because \(\Gamma\) is bounded, the stochastic integral is an
\(L^2\)-martingale and the logarithmic likelihood ratio is integrable.
Since the law of \(X\) is \(\Pp^+\), conditioning the preceding identity
and using the Markov property from
Remark~\ref{rem:girsanov-admissibility} gives, for the fixed Borel Markov
version,
\begin{equation}\label{eq:conditional-entropy-average}
    \frac1h\E_{\Pp^+_{[t,t+h]}}
    \left[
      \left.
      \log\left(\frac{\dd\Pp^+_{[t,t+h]}}
    {\dd\Pp^-_{[t,t+h]}}\right)
      \right|X_t=x
    \right]
    =
    \frac1{2h}\int_0^hP_rg(x)\,\dd r,
    \qquad \mu\text{-a.e. }x.
\end{equation}

It remains to identify the short-time limit. Let \(X^x\) be the mild
solution started from \(x\), and set
\(C:=\sup_{0\leq r\leq1}\|S(r)\|<\infty\). The standard mild-solution
estimate and It\^o's isometry~\cite{DZ92} give, for \(0<r\leq1\),
\[
\begin{aligned}
    \E\|X_r^x-x\|_{\Hh}^2
    &\leq3\|S(r)x-x\|_{\Hh}^2
      +3C^2r^2\|F\|_\infty^2
      +3C^2r\|B\|_{\mathcal L_2(\Uh,\Hh)}^2
      \longrightarrow0.
\end{aligned}
\]
Thus boundedness and continuity of \(g\) yield
\[
    P_rg(x)=\E[g(X_r^x)]\longrightarrow g(x).
\]
Taking the limit in~\eqref{eq:conditional-entropy-average} proves
\[
    \ep(t,x)=\ep(x)
    =
    \frac12g(x)
    =
    \frac12
    \|2F(x)-\Q \nabla\log\rho(x)\|_{\Q^{-1}}^2
    \qquad\text{for }\mu\text{-a.e. }x,
\]
which is~\eqref{eq:ep-density-main}.

Finally, stationarity and the zero mean of the martingale term give
\[
\begin{aligned}
\mathcal H\!\left(
    \Pp^+_{[t,t+h]}
    \,\middle|\,
    \Pp^-_{[t,t+h]}
\right)
&=\frac12\E\int_t^{t+h}g(X_r)\,\dd r
 =\frac h2\int_{\Hh}g(x)\,\mu(\dd x).
\end{aligned}
\]
Dividing by \(h\) and using \(\widehat V=V\) \(\mu\)-almost everywhere
proves
\[
    \EP(t)=\EP
    =
    \frac12
    \int_{\Hh}
    \|2F(x)-\Q \nabla\log\rho(x)\|_{\Q^{-1}}^2
    \rho(x)\nu(\dd x),
\]
and hence~\eqref{eq:EP-stationary-main}.
\end{proof}

\subsection{Proof of the reversibility criterion}

We first record the reverse drift as a consistency check. This calculation
is kept separate from the proof because it does not identify the full
domain of the Hilbert-space adjoint.

\paragraph{Formal \(L^2(\mu)\)-adjoint expression.}
Let \(\varphi\) and \(\psi\) be functions for which
\(\psi\rho\in\Dom(\Lzero)\), \(\psi\rho F\in\Dom(\divnu)\), and the
generator and divergence product rules below are valid. This specifies the
test class needed for the following integration-by-parts calculation; it is
not a claim about the full operator domain of \(L^*\). Symmetry of \(\Lzero\)
and Gaussian integration by parts give
\[
\begin{aligned}
    \int_{\Hh}\psi L\varphi\,\dd\mu
    &=
    \int_{\Hh}\varphi
    \left[
        \Lzero(\psi\rho)-\divnu(\psi\rho F)
    \right]\dd\nu
    \\
    &=
    \int_{\Hh}\varphi
    \Bigl[
        \rho\Lzero\psi
        +\inner{\Q \nabla\psi}{\nabla\rho}
        -\rho\inner{F}{\nabla\psi}
        +\psi\bigl(\Lzero\rho-\divnu(\rho F)\bigr)
    \Bigr]\dd\nu.
\end{aligned}
\]
By~\eqref{eq:FP}, the last parenthesis vanishes. Since
\(\nabla\rho=\rho \nabla\log\rho\) from \eqref{eq:log-gradient-consistency}, it follows that
\[
    \int_{\Hh}\psi L\varphi\,\dd\mu
    =
    \int_{\Hh}\varphi
    \left[
        \Lzero\psi
        +\inner{\Q \nabla\log\rho-F}{\nabla\psi}
    \right]\dd\mu.
\]
Thus, on this test class, the formal adjoint differential expression has
nonlinear drift
\[
    F^{\mathrm{rev}}
    =-F+\Q \nabla\log\rho
    =F-V.
\]
In finite dimensions, if \(p=\rho\gamma\) is the invariant Lebesgue
density, then the total reversed drift satisfies
\[
    Ax+F^{\mathrm{rev}}
    =-(Ax+F)+\Q\nabla\log p,
\]
which agrees with the classical reverse-drift
formula~\cite{HaussmannPardoux1986}. Infinite-dimensional reverse-diffusion
representations under different entropy or Malliavin-regularity hypotheses
were obtained in~\cite{FollmerWakolbinger1986,MilletNualartSanz1989}. The calculation above
is only a consistency check and does not claim to identify the full adjoint
domain or a strong reversed SPDE.

\begin{proof}[Proof of Theorem~\ref{thm:reversibility-criteria}]
\subsubsection*{Step 1. Entropy production, current, and gradient drift}

By~\eqref{eq:EP-stationary-main} and~
\eqref{eq:irreversible-drift},
\[
    \EP
    =
    \frac12\int_{\Hh}\|V(x)\|_{\Q^{-1}}^2\,\mu(\dd x).
\]
The integrand is nonnegative, so
\[
    \EP=0
    \quad\Longleftrightarrow\quad
    V=0\quad\mu\text{-a.e.}
\]
This proves \textup{(i)}\(\Longleftrightarrow\)\textup{(ii)}.

Because \(\rho>0\) \(\nu\)-almost everywhere, \(\mu=\rho\nu\) and
\(\nu\) are equivalent. Moreover, the identity
\(\nabla\log\rho=\nabla\rho/\rho\) from~\eqref{eq:log-gradient-consistency} gives
\[
\begin{aligned}
    J_\mu
    &=
    \rho F-\frac12\Q \nabla\rho
    =
    \frac{\rho}{2}
    \left(2F-\Q \nabla\log\rho\right)
    =
    \frac{\rho}{2}V.
\end{aligned}
\]
Consequently,
\[
    J_\mu=0
    \quad\Longleftrightarrow\quad
    V=0
    \quad\Longleftrightarrow\quad
    F=\frac12\Q \nabla\log\rho
    \qquad \mu\text{-a.e.}
\]
Hence statements \textup{(i)}--\textup{(iv)} are equivalent.

\subsubsection*{Step 2. The pathwise criterion}

The finite-time identity established in the preceding proof gives, for
every \(T>0\),
\begin{equation}\label{eq:finite-time-entropy-reversibility}
    \mathcal H\!\left(
        \Pp^+_{[0,T]}
        \,\middle|\,
        \Pp^-_{[0,T]}
    \right)
    =
    \frac T2\int_{\Hh}\|V(x)\|_{\Q^{-1}}^2\,\mu(\dd x).
\end{equation}
Proposition~\ref{prop:path-ratio} shows that the two path laws are
equivalent. Therefore the equality case in Gibbs' inequality implies
\[
    \mathcal H\!\left(
        \Pp^+_{[0,T]}
        \,\middle|\,
        \Pp^-_{[0,T]}
    \right)=0
    \quad\Longleftrightarrow\quad
    \Pp^+_{[0,T]}=\Pp^-_{[0,T]}.
\]
If \(V=0\) \(\mu\)-almost everywhere, then~
\eqref{eq:finite-time-entropy-reversibility} yields equality of the path
laws for every \(T>0\). Conversely, equality for any one \(T>0\) forces
the integral on the right-hand side of
\eqref{eq:finite-time-entropy-reversibility} to vanish, and hence
\(V=0\) \(\mu\)-almost everywhere. Thus
\textup{(ii)}\(\Longleftrightarrow\)\textup{(vii)}.

\subsubsection*{Step 3. Generator, detailed balance, and path reversal}

We first make  the semigroup fact used here explicit. Invariance of \(\mu\)
and Jensen's inequality give
\[
    \|P_tf\|_{L^2(\mu)}^2
    \leq
    \int_{\Hh}P_t(|f|^2)\,\dd\mu
    =
    \|f\|_{L^2(\mu)}^2.
\]
For \(f\in C_b(\Hh)\), continuity of the stationary sample paths gives
\[
    \|P_tf-f\|_{L^2(\mu)}^2
    \leq
    \E_\mu|f(X_t)-f(X_0)|^2
    \longrightarrow0
    \qquad(t\rightarrow0).
\]
Since \(\Hh\) is separable, \(C_b(\Hh)\) is dense in \(L^2(\mu)\);
contractivity extends the convergence to every \(f\in L^2(\mu)\).
Thus \((P_t)_{t\geq0}\) is a strongly continuous contraction semigroup on
\(L^2(\mu)\).

If its generator \(L\) is self-adjoint, then \(L\) is dissipative and the
spectral theorem, together with uniqueness of the generated semigroup,
gives \(P_t=e^{tL}\); hence every \(P_t\) is self-adjoint;
see~\cite{EngelNagel2000}. Conversely,
detailed balance implies that every \(P_t\) is self-adjoint on
\(L^2(\mu)\). Therefore, for \(f,g\in\Dom(L)\),
\[
\begin{aligned}
    \langle Lf,g\rangle_{L^2(\mu)}
    &=
    \lim_{t\downarrow0}
    \left\langle\frac{P_tf-f}{t},g\right\rangle_{L^2(\mu)}
    \\
    &=
    \lim_{t\downarrow0}
    \left\langle f,\frac{P_tg-g}{t}\right\rangle_{L^2(\mu)}
    =
    \langle f,Lg\rangle_{L^2(\mu)}.
\end{aligned}
\]
Thus \(L\) is symmetric. Since \(L\) generates a contraction semigroup,
the Hille--Yosida theorem shows that every
\(\lambda>0\) belongs to the resolvent set of \(L\). A symmetric operator
with a real point in its
resolvent is self-adjoint, so \(L\) is self-adjoint;
see~\cite{EngelNagel2000}. Therefore
\textup{(v)}\(\Longleftrightarrow\)\textup{(vi)}.

It remains to prove
\textup{(vi)}\(\Longleftrightarrow\)\textup{(vii)}. Assume detailed
balance, fix
\[
    0=t_0<t_1<\cdots<t_n=T,
    \qquad
    \Delta_k:=t_k-t_{k-1},
\]
and let \(f_0,\ldots,f_n\) be bounded real-valued measurable functions.
Write \(M_fh:=fh\). Stationarity and the Markov property yield
\[
\begin{aligned}
    \E_\mu\!\left[\prod_{k=0}^nf_k(X_{t_k})\right]=
    \left\langle
        f_0,
        P_{\Delta_1}M_{f_1}P_{\Delta_2}M_{f_2}\cdots
        M_{f_{n-1}}P_{\Delta_n}f_n
    \right\rangle_{L^2(\mu)}.
\end{aligned}
\]
Both \(P_{\Delta_k}\) and \(M_{f_k}\) are self-adjoint. Taking the
adjoint of the operator product therefore gives
\[
\begin{aligned}
    \E_\mu\!\left[\prod_{k=0}^nf_k(X_{t_k})\right]
    &=
    \left\langle
        f_n,
        P_{\Delta_n}M_{f_{n-1}}P_{\Delta_{n-1}}\cdots
        M_{f_1}P_{\Delta_1}f_0
    \right\rangle_{L^2(\mu)}
    \\
    &=
    \E_\mu\!\left[\prod_{k=0}^nf_k(X_{T-t_k})\right].
\end{aligned}
\]
Thus the forward and time-reversed processes have the same
finite-dimensional distributions. Indeed, an arbitrary finite collection
of times is covered by inserting the endpoints \(0,T\) and taking the
corresponding test functions equal to \(1\). Since \(\Hh\) is separable,
the evaluations at rational times generate the Borel \(\sigma\)-algebra of
\(C([0,T];\Hh)\). Both laws are supported on continuous paths; consequently,
\[
    \Pp^+_{[0,T]}=\Pp^-_{[0,T]}.
\]

Conversely, pathwise reversibility on \([0,t]\), applied to the functional
\(f(X_0)g(X_t)\), gives
\[
    \int_{\Hh}f\,P_tg\,\dd\mu
    =
    \int_{\Hh}g\,P_tf\,\dd\mu
\]
for all bounded measurable \(f,g\). This is detailed balance. Hence
\textup{(vi)}\(\Longleftrightarrow\)\textup{(vii)}, and Steps 1--3 prove
all seven equivalences. This is the standard correspondence between
symmetric Markov semigroups, detailed balance, and reversible stationary
processes; see~\cite[Theorem 4.3.3]{jiang2004}.

Finally, for \(\Phi=-\log\rho\), the gradients satisfy
\(\nabla\Phi=-\nabla\log\rho\) \(\nu\)-almost everywhere. Condition
\textup{(iv)} is therefore equivalent
to
\[
    F=-\frac12\Q \nabla\Phi
    \qquad \mu\text{-a.e.}
\]
This proves the final assertion.
\end{proof}

\section{Conclusion}

This paper establishes a path-space entropy-production formula for a class
of stationary stochastic evolution equations on a separable
Hilbert space. The essential step is to formulate densities, integration by
parts, and probability currents relative to the invariant Gaussian measure
\(\nu\) of a reversible Ornstein--Uhlenbeck reference process. This
circumvents both the absence of an infinite-dimensional Lebesgue reference
measure and the need for a bounded inverse of the trace-class noise
covariance. Combining the Hilbert-space Girsanov transform and the closed Gaussian divergence yields
\[
    \ep(x)
    =
    \frac12
    \|2F(x)-\Q \nabla\log\rho(x)\|_{\Q^{-1}}^2,
\]
and the total entropy production rate is
\[
    \EP
    =
    \frac12
    \int_{\Hh}
    \|2F-\Q \nabla\log\rho\|_{\Q^{-1}}^2
    \rho\,\dd\nu.
\]

The formula identifies
\[
    V=2F-\Q \nabla\log\rho,
    \qquad
    J_\mu=\frac{\rho}{2}V
\]
as the irreversibility field and the stationary Gaussian-reference current,
respectively. On the test class of the formal adjoint calculation, \(V\)
corresponds to the forward-minus-reversed nonlinear drift. Hence
\(\EP\geq0\), and \(\EP=0\) exactly when
\(V=J_\mu=0\) almost everywhere. Under the standing assumptions, this is
equivalent to the gradient relation
\(F=\Q \nabla\log\rho/2\), self-adjointness of the closed generator in
\(L^2(\mu)\), detailed balance, and invariance of the stationary path law
under time reversal.

In finite dimensions, writing \(p=\rho\gamma\) for the invariant Lebesgue
density gives \(V=2b-\Q\nabla\log p\) and
\(j_p=\gamma J_\mu\). The result therefore recovers the nondegenerate
entropy production formula~\cite[Theorem 4.1.7]{jiang2004}, while its minimum-energy norm has the same
range-and-pseudoinverse structure as the degenerate diffusion theory
of~\cite{DaCostaPavliotis2023}. The latter is a structural comparison: under
the present finite-dimensional reversible-reference assumptions, \(\Q\) is
positive definite.

The price of the explicit infinite-dimensional identity is the imposed
Cameron--Martin range and smoothness framework, together with an
\(\Hh\)-valued strong solution. Thus the
framework is more specialized than the general time-reversal results
in~\cite{FollmerWakolbinger1986,MilletNualartSanz1989}. It does not treat the
singular forward--backward regime outside the noise range, familiar from
degenerate finite-dimensional diffusion theory~\cite{DaCostaPavliotis2023},
nor does it construct a strong solution of the reversed SPDE in full
generality. Natural extensions
include replacing the strong-solution argument by a mild-solution
approximation, weakening the differentiability assumptions, and treating
multiplicative noise or a nonreversible Gaussian reference.

\section*{Acknowledgments} 

\section*{Author contributions} All authors made substantial contributions to the analysis and conclusions presented in this work. All authors have read and approved the final manuscript. The author uses AI tools to improve the readability and language of the manuscript. 

\section*{Conflict of interest}
The author declares no conflict of interest.

\section*{Data availability}
No datasets were generated or analyzed in this theoretical study.

\appendix
\section{Gaussian divergence calculation}
\label{app:gaussian-divergence-calculation}

We continue to use the closed Gaussian divergence introduced in
Definition~\ref{def:gaussian-divergence}. This avoids treating
\(Q_\infty^{-1}x\) as an \(\Hh\)-valued random variable: in the genuinely
infinite-dimensional case, a \(\nu=\N(0,Q_\infty)\)-typical point does not
belong to \(\Ran(Q_\infty)\)~\cite{Bogachev1998}. For the closed Gaussian
gradient and divergence, Gaussian integration by parts, and their relation to
Ornstein--Uhlenbeck operators, see~\cite{Bogachev2018, DP04}.

Under Assumption~\ref{ass:AFQ}\textup{(ii)}, \(Q_\infty\) is positive,
trace class, and injective. Hence it has an orthonormal eigenbasis
\(\{e_k\}_{k\geq1}\), 
\[
    Q_\infty e_k=\gamma_k e_k,
    \qquad \gamma_k>0.
\]
Write \(x_k:=\inner{x}{e_k}\), and let \(\partial_k\) denote differentiation in
the direction \(e_k\). Coordinatewise Gaussian integration by parts,
followed by Sobolev approximation, gives
\begin{equation}\label{eq:gaussian-ibp-coordinate}
    \int_{\Hh}\alpha\,\partial_k\beta\,\dd\nu
    =
    -\int_{\Hh}\beta\,\partial_k\alpha\,\dd\nu
    +\frac{1}{\gamma_k}
     \int_{\Hh}x_k\alpha\beta\,\dd\nu
\end{equation}
for \(\alpha\in W^{1,2}(\Hh,\nu)\) and bounded continuously differentiable
cylindrical \(\beta\) with bounded
gradient~\cite[Chapter~5]{Bogachev1998}.

If \(Y=\sum_{k=1}^mY_ke_k\) is a smooth finite-dimensional cylindrical
vector field, applying~\eqref{eq:gaussian-ibp-coordinate} componentwise
yields
\begin{equation}\label{eq:gaussian-divergence-coordinate}
    \divnu Y
    =
    \sum_{k=1}^m\partial_kY_k
    -
    \sum_{k=1}^m\frac{x_k}{\gamma_k}Y_k.
\end{equation}

We now compute the closed divergence of the drift. Exponential stability
implies that \(0\) belongs to the resolvent set of \(A\), with
\[
    A^{-1}=-\int_0^\infty S(r)\,\dd r.
\]
This is the standard Laplace-transform representation of the resolvent of an
exponentially stable semigroup~\cite{EngelNagel2000}.
Under Assumption~\ref{ass:OU}, Remark~\ref{rem:OU-symmetry} gives
\(S(r)\Q=\Q S^*(r)\). Consequently,
\[
\begin{aligned}
    Q_\infty
    &=\int_0^\infty S(r)\Q S^*(r)\,\dd r
      =\int_0^\infty S(2r)\Q\,\dd r =-\frac12A^{-1}\Q.
\end{aligned}
\]
Thus \(Q_\infty\Hh\subset\Dom(A)\), and the bounded-operator identity
\begin{equation}\label{eq:reversible-covariance-identity}
    AQ_\infty=-\frac12\Q
\end{equation}
holds on all of \(\Hh\).

Under Assumption~\ref{ass:path-regularity}\textup{(i)},
\[
    F=\Q\theta,
    \qquad
    \nabla F=\Q\nabla\theta.
\]
Let \(P_m\) be the projection onto
\(\operatorname{span}\{e_1,\ldots,e_m\}\). Since
\(Q_\infty e_k=\gamma_ke_k\) and
\(Q_\infty\Hh\subset\Dom(A)\), one has
\(e_k=\gamma_k^{-1}Q_\infty e_k\in\Dom(A)\); moreover,
\[
    \Q e_k=-2\gamma_kAe_k.
\]
For each \(k\), the scalar function
\(F_k(x):=\inner{F(x)}{e_k}\) belongs to
\(W^{1,2}(\Hh,\nu)\cap L^\infty(\nu)\), and
\[
    \partial_kF_k(x)=\inner{\Q\nabla\theta(x)e_k}{e_k}.
\]
Moreover,
\[
    -\frac{x_k}{\gamma_k}F_k(x)
    =2x_k\inner{A^*\theta(x)}{e_k}.
\]
Set
\[
    d_m(x):=
    \sum_{k=1}^m
    \inner{\Q\nabla\theta(x)e_k}{e_k}
    +2\inner{A^*\theta(x)}{P_mx}.
\]
For every bounded continuously differentiable cylindrical \(\varphi\)
with bounded gradient, summing~\eqref{eq:gaussian-ibp-coordinate} over
\(1\leq k\leq m\) yields
\[
    \int_{\Hh}\inner{P_mF}{\nabla\varphi}\,\dd\nu
    =-\int_{\Hh}\varphi d_m\,\dd\nu.
\]
Here \(P_mF\in L^2(\nu;\Hh)\) and \(d_m\in L^2(\nu)\). Cylindrical test
functions are dense in \(W^{1,2}(\Hh,\nu)\); see~\cite[Chapter~5]{Bogachev1998}.
It follows that \(P_mF\in\Dom(\divnu)\) and
\(\divnu(P_mF)=d_m\).

The trace terms converge uniformly in \(x\):
\[
    \sup_{x\in\Hh}
    \left|
        \Tr(\Q\nabla\theta(x))
        -
        \sum_{k=1}^m\inner{\Q\nabla\theta(x)e_k}{e_k}
    \right|
    \leq
    \|(I-P_m)\Q\|_{\mathcal L_1(\Hh)}
    \|\nabla\theta\|_\infty
    \longrightarrow0.
\]
Moreover,
\(\inner{A^*\theta}{P_m(\cdot)}\to
  \inner{A^*\theta}{\cdot}\) in \(L^2(\nu)\) by dominated convergence,
since
\[
    \bigl|\inner{A^*\theta(x)}{P_mx}\bigr|
    \leq
    \|A^*\theta\|_\infty\|x\|_{\Hh}
\]
and \(\int\|x\|_{\Hh}^2\,\nu(\dd x)<\infty\). Also,
\(P_mF\to F\) in \(L^2(\nu;\Hh)\) because \(F\) is bounded. Closedness of
\(\divnu\) now yields
\begin{equation}\label{divF}
    \divnu F(x)
    =
    \Tr\bigl(\Q\nabla\theta(x)\bigr)
    +2\inner{A^*\theta(x)}{x}
\end{equation}
in \(L^2(\nu)\). We use the continuous function
\begin{equation}\label{eq:divF-borel-representative}
    (\divnu F)(x)
    :=\Tr\bigl(\Q\nabla\theta(x)\bigr)
      +2\inner{A^*\theta(x)}{x},
    \qquad x\in\Hh,
\end{equation}
as the fixed representative of the class in~\eqref{divF}. Since
\(X_r\sim\mu=\rho\nu\ll\nu\), any two \(L^2(\nu)\)-representatives agree
at \(X_r\) almost surely for each fixed \(r\). Stationarity and Fubini's
theorem then show that their time integrals agree almost surely. Throughout
the preceding proofs, every pointwise occurrence of \(\divnu F\) is
understood through~\eqref{eq:divF-borel-representative}.
Formula~\eqref{divF} avoids both the nonexistent bounded inverse of \(\Q\)
and the pointwise expression \(Ax\), which need not be defined for a
\(\nu\)-typical \(x\).

\end{document}